%% file: convergence_generic.tex
\documentclass[a4paper,12pt]{article}
\pdfoutput=1    %pdflatex

\usepackage[sfdefault=cmbr,OMLmathsans]{isomath}  % notation acc. to ISO 31 / 80000
\usepackage{upgreek}  % upright greek letters

\usepackage{tensor}  % for tensor index notation
\usepackage{mathbbol}  %

\usepackage{amsmath}
\usepackage{amsthm}
\usepackage{graphicx}      % include this line if your document contains figures
\usepackage{float}
\usepackage{pst-all}

\newcommand{\figures}{./figures}
\newcommand{\PsTricks}{./PsTricks}

\usepackage{subfigure}% Support for small, `sub' figures and tables

\theoremstyle{plain}
\newtheorem{theorem}{Theorem}[section]

\newtheorem{lemma}[theorem]{Lemma}
\newtheorem{proposition}[theorem]{Proposition}

\theoremstyle{definition}
\newtheorem{definition}[theorem]{Definition}

\theoremstyle{remark}

\usepackage[colorinlistoftodos,disable]{todonotes}
\usepackage{pst-all}
\usepackage{wrapfig}

\usepackage{array}

\newcommand{\vek}[1]{\mathchoice{\displaystyle\boldsymbol{#1}}
{\textstyle\boldsymbol{#1}}{\scriptstyle\boldsymbol{#1}}
{\scriptscriptstyle\boldsymbol{#1}}}

\renewcommand{\vec}[1]{{\vek #1}}%{{\bf #1}}%{{\vek #1}}
\newcommand{\ten}[1]{\tensor[]{ #1}{}}
\newcommand \flux[1]{\vek u_#1}		%Fluss der getauschten Groesse
\def\norm#1{  \left\| #1 \right\|}
\def \pot{\mathcal P}

\def\norm#1{  \left\| #1 \right\|}
\def\scalar#1#2{ \left< #1 , #2 \right>}

\begin{document}

%\jvol{00} \jnum{00} \jyear{2014} \jmonth{February}

%\articletype{GUIDE}

\title{Convergence of explicitely coupled Simulation Tools (Cosimulations)}

\author{Thilo Moshagen$^{\rm a}$ $^{\ast}$\thanks{$^\ast$Corresponding author. Email: t.moshagen@tu-bs.de
\vspace{6pt}} \\ $^{a}${\em{Institut f\"{u}r Wiss. Rechnen,  TU Braunschweig, Germany}};
}%\\ received{March 2017 }
%Institut f\"{u}r Wiss. Rechnen, 
%\tableofcontents
%

\maketitle

\begin{abstract}                % Abstract of not more than 250 words.
In engineering, it is a common desire to couple existing simulation tools together into one
big system by passing information from subsystems as parameters into the subsystems under
influence. As executed at fixed time points, this data exchange gives the global method a
strong explicit component. Globally, such an explicit cosimulation schemes exchange time step can be seen as a step of an one-step method which is explicit in some solution components. Exploiting this structure, we give a convergence proof for such schemes.\\
As flows of conserved quantities are passed across subsystem
boundaries, it is not ensured that systemwide balances are fulfilled: the system is not solved as
one single equation system. These balance errors can accumulate and make simulation results
inaccurate. Use of higher-order extrapolation in exchanged data can reduce this problem but
cannot solve it.
The remaining balance error has been handled in past work 
 by recontributing  it to the input signal in next
coupling time step, a technique labeled \emph{balance correction methods}.
Convergence for that method is proven.\\ 
Further, a proof for the lack of stability of such methods is given for cosimulation schemes with and without balance correction.
\end{abstract}
%\begin{keyword}
\input{introductionEnergy}

There have been some examinations of the convergence of iterative simulator coupling schemes, which can be applied to  noniterative coupling as the case where the number of iterations is 0. But
to the best of our knowledge, no proof for the convergence of the explicite simulator coupling scheme has been presented so far, despite its importance for industrial applications. Neither has there been an examination of convergence of balance correction methods.  This is a considerable gap because the balance correction method is a highly heuristically motivated method with a inherent danger: 
It means making an error in %the amount of a quantity $\int^{T_i+1}_{T_i}y(\tau)d\tau$  that is exchanged between subsystems  during the interval $[T_i,T_{i+1})$,   $u$ being 
the exchanged signal $u$      for lowering the accumulated error in the amount of  that quantity -- it increases the $L1$-error to lower the $L2$-error. This should be considered well.  

\section{Some preliminaries}

%works:$\tensor[]{ A}{}$      %$\ten A$
Always, a time interval or quantities belonging to one are indexed with the index of its right boundary: $\Delta t_i := [t_{i-1},t_i)$. Indexing of times begins with 0.
Big letters refer to time exchange steps, e.g  $\Delta T_k := [T_{k-1},T_k)$ is an interval between exchange times, whereas above interval might denote a subsystems step. Consistently, $k =1,...,N $, and above $i$ ranges from $1$ to $n$ too.
Let $ \operatorname{Ext}(\vec u)_{i}^k(t)$ denote the extrapolation
%\widetilde geht, sieht schietig aus
of the input variable $i$ in interval $k$, %true, not needed: (for more than two subsystems, one would write $ \operatorname{Ext}(\vec u)_{ji}^k(t)$), 
sometimes referred to as \emph{hold function}. 
%true, redundant: some extrapolation of $\vec u$ calculated from the known  $\vec u(T_j)$ at times $T_j<T_k$).

\subsection{Choice of exchanged Variables}
The direction in which the variables are exchanged cannot be chosen arbitrarily. Careful consideration is needed to not place contradictorily constraints on the models. If i.e. at one of the model boundaries $S_1$ exchanges the value of a differential state, it must only export and not receive this value -- otherwise this state is turned into a parameter and extrapolated instead of solved, or one would overwrite an integration result and reinitialize the solver at each exchange. The  choice is to receive arguments of the derivative  of the state and send the state value itself \cite{Scharff12}, \cite{KosselDiss}. Often, the physics of coupling is described by a flow driven by a potential, where the flux depends on the potential by $f\sim \nabla_{\vek x} \pot $, or in non-spatial context   $f = \operatorname{const}\left[\pot_{S_2}- \pot_{S_1}\right]$.  Then exchange is commonly implemented in the following way: System $S_1$ passes the value of the flow-determining potential $\pot$ on to $S_2$, whereas the latter calculates the flux $f $ and passes it to $S_1$.

%\input{recentResearchEnergy}
%using clone of mathmod::\subsection{Detailed description of Techniques in coupling Simulation Software}
\subsection{Detailed description of Techniques in coupling Simulation Software}
\label{sec-2-1}
After the  brief overview in the introduction on the issues that appear when coupling simulation software and means to treat them was given. Those fields are described more precisely here, giving citations. 

\subsubsection{Determination of consistent Inputs}\label{sec_consistentInput}
As $\vek u_i$ fulfils $\vek 0 =\vek h_i(\vek x_i, \vek z_i, \vek u_i)$, where $\vek z_i$ results from  $\vek 0 =\vek g_i(\vek x_i, \vek z_i, \vek u_j)$, it in general depends on $\vek u_j$. So the coupling equations \eqref{coupling1} \eqref{coupling2} $(\vec h_i)_{i=1..n}=\vek 0$ are a coupled system in fact.
 The $\vek 0 =\vek h_i$ can be solved with respect to the $\vek u_i$ one after the other  only if the directed graph of  influence of outputs on each other contains neither bidirectional dependencies nor loops. Otherwise, the  coupling conditions remain not fully fulfilled, or one solves them iteratively. Such
  methods are the Interface-Jacobian based methods,  e.g. given in \cite{KueblerSchiehlen2000},  which determine the $\vec u_i$ with respect to $\vec h_i = 0$ and $\vec g_i = 0$ with the help of a newton solver at exchange times, which means full consistency at those times. But not all  tools simulating the subsystems provide the residual of  $\vek g_i=0$.
%\end{enumerate}

\subsubsection{Iterative coupled Methods}
 If restarting the timestep is an option, iterations can be performed and instead of using $\operatorname{Ext}(\vek u_i)$, the time-dependant numerical solution for $\vec u_i$ from the old iteration can be used.
 Schemes are distinguished by the exchange directions of input,
 as the waveform iteration or the Gauss-Seidel scheme, which can exploit
  the directions of influences between the subsystems (\cite{MiekkalaNevanlinna87}, \cite{ArnoldGuenther2001}, \cite{Busch2012}).
  Stability and accuracy, also in terms of balance, such can be augmented, of course at computational cost due to the restarts.\\ 
In \cite{SicklingerBelskyEngelmannElmqvist14},  a Newton method is used to solve the coupling equations  \eqref{coupling1}-\eqref{coupling2}, but different from \cite{KueblerSchiehlen2000} not only $(\vec h_i(\vek x_i(T_k), \vek z_i, \vek u_i))_{i=1..n}=\vek 0$ at the end of the timestep, but   $(\vec h_i(\vek x_i(\vek u_i), \vek z_i, \vek u_i))_{i=1..n}=\vek 0$. This requires repeated solving of the subsystems alone for Jacobian evaluation.
 
\subsubsection{Increasing the extrapolation order of inputs} To improve approximation $  \operatorname{Ext}\vec u^i$ has been calculated as an extrapolation polynomial of past
   $\vec u^{j}$, $j<i$. 
In his dissertation \cite{Busch2012} Busch examines Lagrangian and Hermite extrapolation polynomials for exchanged quantities using a system concatenated from two coupled spring-mass oscillators, thus a linear  second order ODE with four degrees of freedom,  as model problem.

\subsubsection{Investigations of convergence}
Busch examines  the effect of the approximation order in exchanged data on global convergence and judges the stability of the method for his problem, assuming exact integration of each subsystem. %The examination of convergence is executed %For zero order  and for linear extrapolation %with the result that
 Convergence to exact solution in this case is limited to $O(h^2)$ for piecewise constant extrapolation, while it is limited to $O(h^3)$ for degree one Hermite polynomial extrapolation.\\
%In is dissertation \cite{Busch2012},  Busch investigates the convergence of a cosimulation scheme applied to a two-Spring-two-mass using constant and linear extrapolation of inputs. %Progress has been made on the analytical investigation of the convergence and stability, e.g. \cite{Busch2012} calculates relationships between quality of exchanged data and convergence order as well as stability and 
%For this problem, he proves that convergence order of the co-simulation  increases with extrapolation order . 
As Busch examines a system of ordinary differential equations without algebraic parts and given explicitely, he feels no need to consider smoothing. 
Balance errors that occur in his problem remain undiscussed. 
Besides the aforementioned examination \cite{Busch2012} which is limited to linear problems,  in \cite{ArnoldGuenther2001}, similarly in \cite{MiekkalaNevanlinna87}, there is a more general examination, treating iterative coupling schemes. Mainly, a fixed point argument is used here.

\subsubsection{Smooth switching between old and new input }
The smoothing of exchanged data is originally motivated by the need to  provide a starting value close to the solution during searching for solutions of the algebraic  part of the differential-algebraic equation system. %When parameters of the algebraic system change continuously in time, so will the solution of that system. 
\\
Another reason for smoothing the inputs can be that it enables the calculation of derivatives resp. difference quotients from them. Although this is unfortunate, such needs sometimes occur in practice.\\
In  \cite{KosselDiss}, two halves  of a cooling cycle are co-simulated, the mass flow $\flux{m}$ being exchanged, so a component of $\vek u$.  The smooth input $\overline{  \vek{ u}}_{i}^{j}$ is concatenated as a convex sum of $\operatorname{Ext}{ \vek{ u}}_{i}^{j-1}$ and $\operatorname{Ext} \vek{ u}_{i}^{j}$ weighted
with a sufficiently smooth, here degree 5 polynomial, function switching from 0 at $T_{j-1}$ to 1 at $T_{j}$.
%Smoothness in $\overline{ \appFluxi{m}{i}}$ is achieved. 
Besides from stabilizing the Newton solver as inteded, Kossels work gives an impression on balance errors: During some simulations balance error of about 1.25\% in coolant mass is observed.\\
It is obvious
that by this smoothed switching %between $\appFluxi{m}{i-1}$ and $\appFluxi{m}{i}$ 
the balance error
 $\Delta E_{i}^j$ possibly increases compared to a
unsmoothed switching procedure.\\ 
Another work that considers smoothing is the paper \cite{WellsHasanLucas}.
%\subsection{}
% The work \cite{WellsHasanLucas} aims at putting together old and new extrapolation too and in fact Kossels smoothing methods can be seen as one way of implementing the very generally described method given in this paper. 
%In this work exchanged quantities  are approximated by a sum of an extrapolation polynomial and a polynomial function that establishes the smooth change towards the new extrapolation ( called predictor here). 
%The authors give a recipe on how polynomial smooth switching function should be constructed such that the sum of correction and extrapolation is continuously differentiable up to a certain differentiation order. To this end an equation system consisting of those smoothness restrictions has to be solved. The effect of implementations of various smoothness on a cosimulation is examined.

\subsubsection{Balance Correction}\label{sec:BC}
%\subsubsection{Balance Errors} 
%Under piecewise constant extrapolation of exchanged data, intervals with monotonely rising $f_m$,  equivalently, $d_t f_m \ge 0$, yield a negative contribution to $\Delta E_{m}$ and  intervals with monotonely falling $f_m$, equivalently, $d_t f_m \le 0$, give rise to a positive balance error. Analogously, using linear interpolation,  intervals with convex $f_m$,  equivalently, $d_t^2 f_m \ge 0$, yield a positive  and  those with %intervals with concave $f_m$, equivalently, $d_t^2 f_m \le 0$ give rise to a negative balance error. 
%In generalized terms: 
Classical balance errors of extrapolations is  $\Delta E_{i}^k = 
\int_{T_{k-1}}^{T_k}\vek u_{i} - \overline{  (\vec u)}_{i}^k(t) dt$ is %the error of the quantity $m$, thus 
where $\vek u_{i}$ is a flux of a conserved quantity, but it is defined for arbitrary quantities. %Such errors
%of $k$-th order arise on intervals where $d_t^{k+1} f_m$, $f_m$ being a conserved flow, does not change signs. %, see figure \ref{monotoneInterval}.
 Negative and positive contributions from positive and negative intervals partly compensate each other, and 
as in a typical simulation the system often ends up in a stationary state similar to the one at begin, thus graphs of quantities and their derivatives tend to end at values where they started, the conserved quantities balance error at the end of a simulation may be small. This does not imply it is small during all intervals of the simulation. As a typical such simulation situation think of an automotive driving cycle (\cite{KosselDiss}): Using piecewise constant extrapolation of $f$, there  is a gap in mass after the phase of rising system velocity has passed - this loss remains uncompensated during the (significant) phase of elevated speed.\\
Scharff, motivated by the loss in balance stated in \cite{KosselDiss} as a collateral result, proposed in \cite{Scharff12} that the errors 
\begin{equation}\label{eq:BC}
\Delta E_i^{j} = \int_{T_{j-1}}^{T_j}\vek u_i dt  - \int_{T_{j-1}}^{T_j} \overline{ \vek u}_i dt
\end{equation} 
are
added in the time step $j+1$.  
The correction that is applied in  the $j$-th interval is then
%\begin{equation} 
$\Delta E_i^{j-1}\phi_j (t)$,
%\end{equation}
where   the function $\phi_j$ is scaled such that its integral is 1 and  may be constant but when one wants to  preserve smoothness, it is a smooth function smoothly vanishing
at the boundaries of the time interval $j$. 
In spite of the
lack of strictness and the increase of errors in derivatives of the
exchanged quantities, this method enables coupling of simulations that
would be impossible without - Kossels example was recalculated successfully using balance correction. 
%end clone mathmod::\subsection{Detailed description of Techniques in coupling Simulation Software}
Balance correction methods were applied to  nonconserved quantities  in \cite{ScharffMoshagen2017}, as even such quantities have some conservation properties in space and time, and examination will go on here. The method even improved the convergence for  a linear  such problem.\\

\subsubsection{Reducing exchange induced derivations in input signals}
Split systems may have parts that are prone to be excited by quick changes, implying high derivatives, in input signals.
An example for such a system is a high-frequency spring-mass system sitting on a slowly moving ground whose displacement is given to the spring as input. In such settings it is essential to avoid the contribution of additional frequencies to the signal that the especially the balance corrected coupling method does. In \cite{ScharffMoshagen2017}, it was suggested to spread recontribution of balance correction contributions over more than one interval.

\subsection{Aim of this work}

Most part of our work can be understood as a contribution to Arnold and 
Guenthers theoretical examination and convergence results.
\cite{ArnoldGuenther2001}, but putting their result about convergence \cite[Th. 2.3]{ArnoldGuenther2001} into the context of the standard explicit cosimulation scheme from Section \ref{sec:scopeNotation} %table \ref{tab:generalScheme} 
that 
is paying respect to practical issues: To the need to run  subsystems  simultaneously by  applying a  Jacobi instead of the Gauss-Seidel scheme there, and to the need not to be intrusive into subsystems methods by avoiding iteration.
Our convergence results will turn out to be consistent to those presented in \cite{ArnoldGuenther2001}, but to require simpler derivation as no iterations occur and offer insight into subsystem methods errors effect.  \\

\input{convergenceAndStability}

%\input{classificationOfBilanceErrors.tex} 

\input{convergenceOfBalanceCorrected.tex}

\input{numResultsEnergyI.tex}

\section{Discussion and Conclusion}% and future work}

%\subsection{Conclusion} 
\label{sec:conclusion}
By the estimates given in  section \ref{convergence}, theorem \ref{th:convergence} %estimate \eqref{convErrBoundingFunction}  
 %equation \eqref{convergenceBC}
 cosimulation methods are   proven to be efficient methods, if the overall system is not too vulnerable to stability problems (see Section \ref{Stability}).
The stability issue will be tackled in a separate publication.\\
The convergence result is consistent with those in \cite{ArnoldClaussSchierz2013} and in \cite{ArnoldGuenther2001}, but include the error of the subsystems methods. 
 Moreover, with the provided methodology it can be proven that balance correction method converges at least with the same convergence order as the cosimulation scheme without balance correction, see theorem \ref{th:convergenceBC} from  section \ref{sec:convergenceBC}. For most problems it should be the case that  convergence order  rises by one if  balance correction is applied, as the test problems and superficial considerations indicate.\\
 Finding a  sharper theoretical result for balance correction techniques seems a realistic future task.

%\bibliography{../paper/ifacconf}
%\bibliographystyle{nMCM}
\providecommand{\bysame}{\leavevmode\hbox to3em{\hrulefill}\thinspace}
\providecommand{\MR}{\relax\ifhmode\unskip\space\fi MR }
\providecommand{\MRhref}[2]{%
  \href{http://www.ams.org/mathscinet-getitem?mr=#1}{#2}
}
\providecommand{\href}[2]{#2}

\end{document}

%% file: introductionEnergy.tex
\section{Introduction}
\label{sec-1}

Engineers are increasingly relying on numerical simulation techniques. Models and 
simulation tools for various physical problems have come into existence in the past 
decades. The desire to simulate a system that consists of well described and treated 
subsystems by using appropriate solvers for each subsystem and letting them exchange the data that forms the mutual 
influence is immanent. \\ 
The situation usually is described by two coupled differential-algebraic systems $S_1$ and $S_2$ that together form a system $S$:
\begin{align}
S_1: \quad \nonumber \\ 
\dot{\vek x}_ 1 &= \vek f_1(\vek x_1,\vek x_2,\vek z_1, \vek z_2)\\
0 &= \vek g_1(\vek x_1, \vek x_2, \vek z_1, \vek z_2) \\
S_2: \nonumber \quad\\
\dot{\vek x}_2  &= \vek f_2(\vek x_1,\vek x_2,\vek z_1, \vek z_2)\\
0 &= \vek g_2(\vek x_1, \vek x_2, \vek z_1, \vek z_2). 
\end{align}
The $(\vek x_1,\vek x_2)$ are the differential states of $S$, their splitting into $\vek{ x}_i$ determines the subsystems $S_i$ together with the choices of the $\vek{z}_i$.
In \emph{Co-Simulation} the immediate mutual influence of subsystems 
 is  replaced by exchanging data at fixed time points 
 and subsystems
are solved separately and parallely but using the received parameter:
\begin{align}
S_1: \nonumber \quad\\
\dot{\vek x}_ 1 &= \vek f_1(\vek x_1, \vek z_1, \vek u_{12})\label{split1} \\ 
0 &= \vek g_1(\vek x_1,  \vek z_1, \vek u_{12}) \label{splitAlgebraic1} \\ 
S_2:\quad \nonumber \\
\dot{\vek x}_2 &= \vek f_2(\vek x_2,\vek z_2, \vek u_{21})\\
0 &= \vek g_2(\vek x_2, \vek z_2,  \vek u_{21})  \label{splitLast}
\end{align}
where $\vek u_i$ are given by coupling conditions that have to be fulfilled at exchange times $T_k$ 
\begin{align}
\vek 0 &=\vek h_{21}(\vek x_1, \vek z_1, \vek u_{21})\label{coupling1}  \\ 
\vek 0 &=\vek h_{12}(\vek x_2, \vek z_2, \vek u_{12})  \label{coupling2}
\end{align}
and are not dependent on subsystem $i$'s states any more, so are mere parameters between exchange time steps.\\
Full row rank of $d_{\vek z_i} \vek g_i$ can be assumed, such that the differential-algebraic systems are of index 1. This description of the setting is widespread (\cite{ArnoldGuenther2001}). 
With the $\vek h_{ij}$ being solved for $\vek u_{ij}$ inside the $S_j$ (let solvability be given), for systems with more than two subsystems it is more convenient to write  output variable $\vec y_j$  %instead of $\vek u_i$   
and now redefine $\vek u_{ij}$ as the input of $S_i$, consisting of some components of the outputs $\vek y_j$ %(\vec y_i)_{i=1..n} $
\cite{ArnoldClaussSchierz2013}. This structure is defined as kind of a standard for connecting simulators for cosimulation  by the \emph{Functional Mockup Interface} Standard \cite{FMI}. It defines clearly what information a subsystems implementation provides. \\ %Mind input $\vek u_j$ is then indexed with its subsystems index, which we switch to from Chapter 2 on.\\
In %explicit 
Co-Simulation the variables establishing the mutual influence of subsystems 
 are exchanged at fixed time points.
 This results in continuous variables being approximated by piecewise constant
  extrapolation, as shown in the following picture: \\
\begin{figure}[h!tb]
\begin{center}
\resizebox{0.5 \textwidth}{!}{
\includegraphics{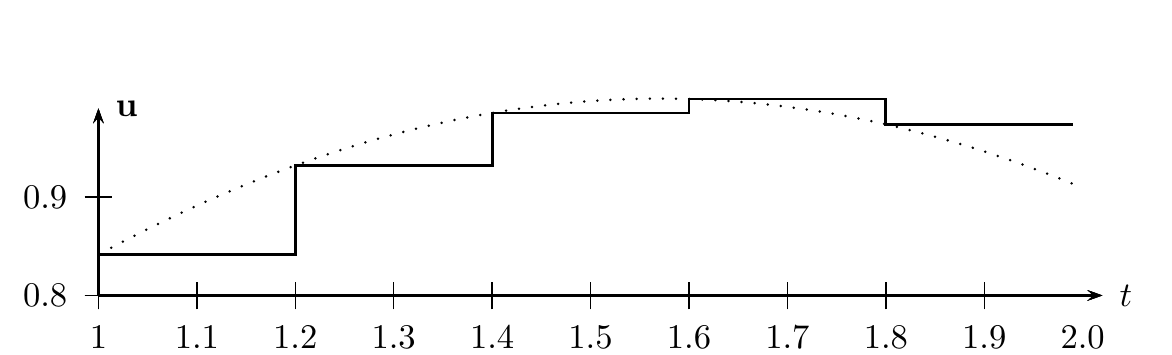}
}\end{center}
\caption{Constant extrapolation of an input signal}
\end{figure}

If one does not want to iterate on those inputs by restarting the simulations using the newly calculated inputs, one just proceeds to the next timestep. \\
This gives the calculations an explicite component, the mutual influence is now not immediate any more, inducing the typical stability problems, besides the
%Although the mutual influence is now not immediate any more and subject to
  approximation errors.\\
 \begin{figure}[h!tb]
\begin{center}
\resizebox{0.7\textwidth}{!}{
\includegraphics{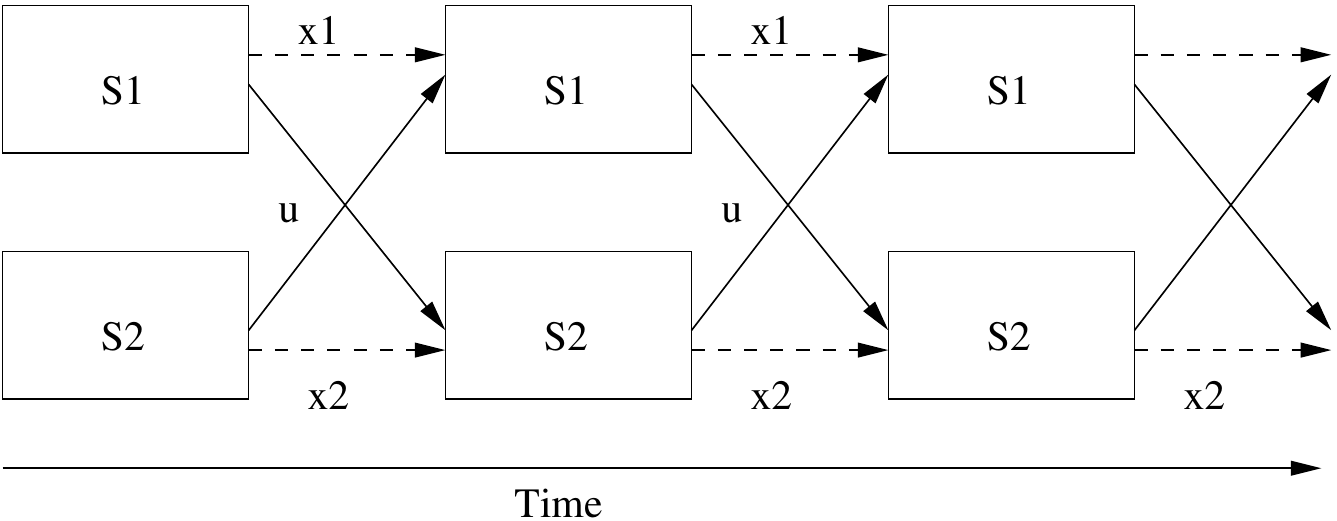}
}\end{center}
\caption{Explicit Cosimulation Scheme}
\end{figure}
But for good reasons, explicit co-simulation is a widely used method: 
It allows to put  separate submodels, for each of which a solver exists, together into one system  and simulate that system by simulating each subsystem with its   specialised solver - examination of mutual influence becomes possible without rewriting everything into one system, and   simulation speed benefits from the parallel calculation of the submodels. Usually it is highly desirable that a simulation scheme does not require repeating of exchange time intervals or iteration, as for many comercial simulation tools this would  already require  too deep intrusion into the subsystems method and too much programming in the coupling algorithm.\\

%\sout{Valid principles have been determined on how subsystem boundaries should be chosen and what quantities  should be exchanged \cite[]. In brief, if the flux of a quantity is passed to one subsystem, this subsystem should pass the value of the potential determining that flux in opposite direction. \\}
The following fields of work on explicite co-simulation can be named to be the ones of most interest:
\begin{enumerate}
\item Improvement of the approximation  of the exchanged data will most often improve simulation results \cite{Busch2012}. This  is usually done by \emph{higher-order extrapolation} of exchanged data, as shown in this plot, where  the function plotted with dots is linearly extrapolated:\\ 
\begin{figure}[h!tb]
\begin{center}
\resizebox{0.5 \textwidth}{!}{
\includegraphics{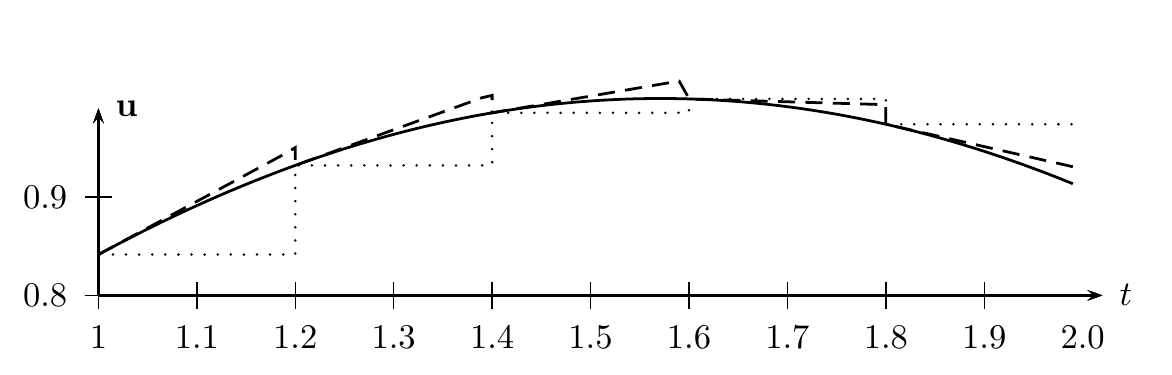}
%\begin{pspicture}(10,7.5)(21,11)
}\end{center}
\caption{Linear extrapolation of an input signal}
\end{figure}
%
%The picture above shows a first order extrapolation of the function plotted with dots. The error with extrapolation obviously smaller. 
%Progress has been made on the analytical investigation of the convergence and stability, e.g. \cite{Busch2012} calculates relationships between quality of exchanged data and convergence order as well as stability and proves that convergence order of the co-simulation can be increased. 

\item When the mutual influence between subsystems consists of flow of conserved quantities like mass or energy, it turns out that the improvement of the approximation of this influence by extrapolation of past data is not sufficient to establish the conservation of those quantities with the necessary accuracy.
The error that arises from the error in exchange adds up over time and
 becomes obvious (and lethal to simulation results many times). 
 In a cooling cycle example (\cite[Section 6.3]{KosselDiss}),  a gain of 1.25\% in coolant mass occurs when simulating a common situation.% in a simulated quarter of an hour. \\
\begin{figure}[h!tb]
\begin{center}
\resizebox{0.5 \textwidth}{!}{
\includegraphics{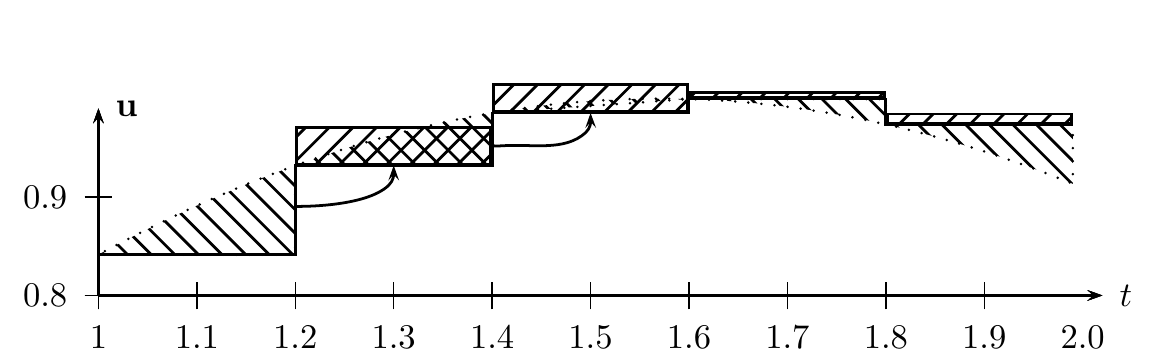}
}
\end{center}
\caption{Constant extrapolation of an input signal, balance error and its recontribution}
\end{figure}
It has been tried  to meet this challenge by   passing the amount of exchanged quantity for the past timestep along with the actual flow on to the receiving system, where then the error that has just been  commited is calculated  and added to the current flow to compensate the past error. For well damped example problems in fluid circles this method has fulfilled the expectations \cite{Scharff12}. It has been labelled \emph{balance correction}. 
\item There is good reason to prevent jumps in exchanged data by \emph{smoothing}. Higher order extrapolation polynomials cannot make extrapolated  data at the end of the exchange timestep match the newly given value.
\end{enumerate}

%% file: convergenceAndStability.tex
\section{Convergence and stability for One-step methods on subsystems } \label{convergenceAndStability}
%geeigneter Ort, Subsection
Cosimulation schemes without iteration (repeating timesteps) can be regarded as mixed explicit-implicit methods.\\
We thus introduce some notations and facts from  ODE solving first, mostly the way it is done in \cite{DeuflhardBornemann94}, before we consider cosimulation methods.

\subsection{Scope and notation}\label{sec:scopeNotation}
The ODE approach that is applied here covers the setting introduced by equations \eqref{split1} -  \eqref{splitLast}:  As $\partial_{\vec u} \vec h $ has full rank, equations  
\eqref{coupling1} and \eqref{coupling2}
 can be solved by $\vec u_i(\vec x_i)=\vec h_i^{-1}(\vec x_i)$. Further,  by applying the \emph{state space method}, which is solving the algebraic equations \eqref{splitAlgebraic1} and \eqref{splitLast} at each evaluation of $\dot{ \vec x}$, the problem is shifted to the solving of the ODE 
\begin{align}
S_1: \quad 
\dot{\vec x}_ 1 &= \vec f_1(\vec x_1, \vec g_1^{-1}(\vec x_1, \vec u_2), \vec u_2)\\%??? \vec u_2(\vec x _2)), \vec u_2(\vec x _2))  \\ 
S_2:\quad 
\dot{\vec x}_2 &= \vec f_2(\vec x_2, \vec g_2^{-1}( \vec x_2, \vec u_1), \vec u_1)%???\vec u_1(\vec x_1)), \vec u_1(\vec x_1)).
\end{align}
The cosimulation scheme for two ODE subsystems whose coupling equations \eqref{coupling1} -\eqref{coupling2} shall be given explicitely by $\vec u_i=\vek h_i^{-1}(\vec x_i)$ instead of $\vek 0 =\vek h_i(\vec x_i,\vec u_i)$, reads:
%\begin{table}[h!tb]\label{tab:generalScheme}
\begin{center}
\begin{tabular}{p{0.27\textwidth} | p{0.27\textwidth}} 
\multicolumn{2}{ c }{General scheme}\tabularnewline
\center{$S_1$} & \center{$S_2$}\tabularnewline
%\begin{equation}
\multicolumn{2}{ c }{System States}\tabularnewline
\center{$\vec x_1  $} 		&\center{$\vec x_2 $} \tabularnewline
%\multicolumn{2}{ c }{Outputs}\tabularnewline
%\center{$\vec y_1  $} 		&\center{$\vec y_2 $} \tabularnewline
\multicolumn{2}{ c }{Inputs}\\
\center{$\vec u_2=\vek h_2^{-1}(\vec x_2)$}	& \center{$\vec u_1=\vek h_1^{-1}(\vec x_1)$}\tabularnewline
\multicolumn{2}{ c }{Equations}\tabularnewline
\center{$\dot {\vec x_1} = \vec f_1(\vec x_1, \operatorname{Ext}(\vec u_2))	$} & \center{$\dot{ \vec x_2} = \vec f_2(\vec x_2, \operatorname{Ext}(\vec u_1))$}	\\
\end{tabular}\\
\end{center}
%\end{table}
%
By shifting the evaluation of $\vec u_i=\vek h_i^{-1}(\vec x_i)$ to the receiving subsystem $S_j$, $\vec u_i = \vec x_i$ can be written in the $S_1-S_2$ system (this straightforwardly generalizes to more subsystems)% fehl am platz: e.g. by writing $\vec u_{ij}$, from $j$ into $i$)
, and a  notation is achieved that is better suitable for applying techniques from ODE.  %the variables $\vec x$ of the subsystem under consideration are split into the subsystems differential states, whose 
Moreover, for readability 
the subsystems differential states
set of indices shall now be $\mathcal D$, and that of the   input variables, which shall be  differential states of other subsystems, %, whose set of indices 
shall be $\mathcal I$. Thus e.g. $\vec x_{ \mathcal I}:=\vec x_2$ in first subsystem, and $\vec x = \left[{\vec x}_{\mathcal D},\vec x_{\mathcal I}\right]$ and the ODE governing the subsystem can be written as 
\begin{equation}
\dot{\vec x}_{\mathcal D} = f(t,\left[{\vec x}_{\mathcal D},\vec x_{\mathcal I}\right] ).
\end{equation}
In cosimulation methods, input variables $\vec x_{\mathcal I}$ are extrapolated from past data into the current timestep. This extrapolation is denoted $\operatorname{Ext}(\vec x_{ \mathcal I}) $.\\

\subsection{Consistency}

Let $\left\{t_j\right\}$, $j=1,...n$ be the time steps taken by the subsystems method in one step. It is $h_j = t_{j+1} - t_j$. \\
As in \cite{DeuflhardBornemann94}, the \emph{evolution} $\Phi^{t+\tau, t}x$ of some ODE through $x$ at $t$ until $t+\tau$  is the solution of the IVP $\dot{ \vec x }= f(t,\vec x)$, $\vec x(t)= \vec x$. Similarly, the \emph{discrete evolution} $\Psi^{t+\tau, t}x $ is the numerical solution  of that IVP produced by a given method, $\Psi^{t+h, t} $ thus the evolution on one time step.
 
The following common definitions are used:
\begin{definition}[Consistency]\label{defConsistency}
\begin{equation}
\epsilon(t,\vec x, h) = \Phi^{t+h, t}\vec x - \Psi^{t+h, t}\vec x
\end{equation}
is called the consistency error of a one-step method.
A discrete evolution $\Psi^{t+h, t}x$ has \emph{consistency order} $p$ if
\begin{equation}
\epsilon(t,\vec x, h) = O(h^{p+1}) \qquad \text{for} \quad h\longrightarrow 0.
\end{equation}
\end{definition}
The following is given in brief, indeed it is a consequence of the two consistency criteria $ \Psi^{t,t}\vec x = x$ and  $\frac{d}{d\tau}\Psi^{t + \tau,t}x \mid_{\tau=0}= f(t,\vec x)$.

\begin{lemma}
Let $\Psi^{t + h, t}\vec x $ be differentiable with respect to $h$. Then the following is equivalent:
\begin{itemize}
\item The discrete evolution $\Psi^{t + h, t}\vec x $ is consistent.
\item The discrete evolution can be written as $\Psi^{t + h, t}\vec x = \vec x + h \psi(t,\vec x, h)$ where $\psi$ is called the increment function and is continuous w.r.t $h$.
\item $\epsilon(t,\vec x, h) = o(h)$ for $h\longrightarrow 0$, i.e. the consistency error vanishes near $t$.
\end{itemize}
\end{lemma}
\begin{proof} See \cite[lemma 4.4]{DeuflhardBornemann94}
\end{proof}
\subsection{Convergence}
We call the set of time points between which one step of the method is executed the \emph{grid}: $\Delta:=\{t_i\}$. Each one-step method defines a \emph{grid function} or numerical solution $\vec x_\Delta$ by solving the IVP $\dot{ \vec x }= f(t,\vec x)$, $\vec x(t_i)= \vec x_\Delta(t_i)$ on $t_i, t_{i+1}$, in other words, by the recursion
\begin{align}
\vec x_\Delta(t_0) &= \vec x_0\\
\vec x_\Delta(t_i) &= \Psi^{t_i,t_{i-1}}\vec x_\Delta(t_{i-1}).
\end{align}
\begin{definition}[Convergence]
The \emph{convergence error} of a discrete solution $\vec x_\Delta$ is defined as
\begin{equation}
\epsilon_{\Delta }(t) = \vec x(t)- \vec x_\Delta(t)
\end{equation}
A discrete solution $\vec x_\Delta$ has \emph{convergence order} $p$ if
\begin{equation}
\epsilon_{\Delta }(t) = O(h^p) \qquad \text{for} \quad h\longrightarrow 0.
\end{equation}
\end{definition}

The following result for ODEs is well-known (see, e.g., \cite{DeuflhardBornemann94}, theorem 4.10:)
\begin{theorem}\label{consistencyToConvergence}
Let the consistency error of an evolution  $\Psi^{t + h, t}\vec x = \vec x + h \psi(t,\vec x,h)$ given by an one-step method with $\psi$ Lipschitz-continuous in $\vec x$ satisfy 
\begin{equation}
\epsilon(t,\vec x, h) = O(h^{p+1}) 
\end{equation}
Then for all grids $\Delta$ with sufficiently small width $h$ the evolution defines a numerical solution $\vec x_\Delta$ for the initial value $\vec x_0=\vec x(t_0)$. The numerical solution converges of order $p$ towards $\vec x$, equivalently,
\begin{equation}
\epsilon_{\Delta }(t) = O(h^p). 
\end{equation}
\end{theorem}
In short, the theorem says that if a method has consistency order $p$, it follows that it is also convergent of order $p$.

\subsection{Convergence of Cosimulation schemes}
\label{convergence}
Now times of data exchange are denounced as $\left\{T_i\right\}$, $i=0,...m$, the exchange step width as $H_{i+1}= T_{i+1} - T_i$, whereas the $\left\{t_j\right\}$, $j=0,...n$  are the time steps taken by the subsystems method, counting beginning again with 0 in each timestep (indexing by macro and local timestep is not needed). It is $h_{j+1} = t_{j+1} - t_i$. For convienience, it is assumed that $H_i=\sum_{j=1}^n h_j$. \\
As in \cite{DeuflhardBornemann94}, the \emph{evolution} $\Phi^{t+\tau, t}x$ of some ODE through $x$ at $t$ until $t+\tau$  is the solution of the IVP $\dot{ \vec x }= f(t,\vec x)$, $\vec x(t_0)= \vec x$. 
%\todo{Joachim sagt: gilt fuer alle expl. Verfahren}
\begin{proposition} There exists a consistent and thus convergent cosimulation scheme, given that derivatives for all subsystems are lipschitz-continuous for all subsystems.
\end{proposition}
 This holds for the following: If Euler-forward scheme with one step  between exchange times $T_i$  is executed, this is equivalent to executing Euler-forward on the whole system $S$ and thus, for Euler-forward being consistent of order 1, this cosimulation method is convergent of order 1 due to theorem \ref{consistencyToConvergence}. \\
 
Motivated by this simple result, we try to transfer the methods used to prove Theorem \ref{consistencyToConvergence} to realistic subsystem methods as convergence should improve with use of higher-order subsystem methods and higher-order extrapolation of exchanged data. The proof also sketches which steps are to be taken:
\begin{enumerate}
%ersetzt:\item  \todo{Redundant? Wenn (2)=>ok. Fuer 2 benoetigt?  Wo?}Consistency on the subsystems for step $[t_0,t_1)$ is given \todo{nicht ganz trivial, ist naechster abschnitt}by the consistency of the subsystems methods.
\item  Consistency on the subsystems for step $[T_0,T_1)$ is nearly trivially given only  for $[t_0,t_1)\subset [T_0,T_1)$ by the consistency of the subsystems methods. 
For the other subintervals $[t_{i-1},t_i)\subset [T_0,T_1)$ a consistency-like estimate has to be proven --  due to the contributions from extrapolated exchanged data the convergence  on $[T_0,T_1)$ can not be concluded from consistency of subsystem methods using theorem \ref{consistencyToConvergence} in an obvious way.
\item If an estimate on  $[t_{i-1},t_i)$ is given, convergence of subsystems methods on $[T_0,T_1)$ can be proven in a similar way as theorem  \ref{consistencyToConvergence}. This then  means  consistency of a step of the cosimulation method on the whole system $S$ is provided. 
\item If by (2) consistency of a step  $[T_0,T_1)$ of the cosimulation method is provided, again by applying theorem \ref{consistencyToConvergence} to cosimulation steps global convergence of the cosimulation follows.%\todo{nicht soo einfach}
\end{enumerate}

\begin{figure}
\setlength{\fboxsep}{7pt} 
Standard One-Step :\\
\framebox{\parbox{0.28\textwidth}{ 
\small{split $\epsilon_\Delta$ error of $\vek x_\Delta$
 into consistency error  and propagated error}}}
$\rightarrow$\framebox{\parbox{0.31\textwidth}{\small{by consistency, recurrence estimate  $ \norm{\vec\epsilon_{\Delta}(t_{j+1})} \le   Ch^{p+1} + (1 + h_jL)\norm{\vec\epsilon_{\Delta}(t_{j})}$ emerges}}}
$\Rightarrow$\framebox{\parbox{0.18\textwidth}{\small{convergence}}}\\[10pt]
partly $H$-explicit Method:\\
\framebox{\parbox{0.28\textwidth}{ 
\small{split $\epsilon_\Delta$  on $[t_{j-1},t_j]\subset[T_0,T_1]$
 into four parts}}}
$\rightarrow$\framebox{\parbox{0.31\textwidth}{\small{by $h-H$ consistency-like estimate,  
similar recurrence estimate for $\epsilon_\Delta$ emerges}}}
 $\Rightarrow$\framebox{\parbox{0.18\textwidth}{\small{convergence on $[T_0,T_1]$}}}\\
%$\Rightarrow$\framebox{\parbox{0.25\textwidth}{convergence}}\\
$\Leftrightarrow$\framebox{\parbox{0.25\textwidth}{\small{consistency of cosim step}}}
$\Rightarrow$\framebox{\parbox{0.31\textwidth}{\small{by consistency, similar recurrence estimate emerges}}}
$\Rightarrow$\framebox{\parbox{0.18\textwidth}{\textbf{convergence}}}\\
\caption{\label{fig:SketchOfConvergence}Sketch of convergence proof. Above, standard proof for one-step methods for comparison.}
\end{figure}
We have to show items (1) and (2). 
\\

\subsubsection{Consistency of Subsystems steps inside $\left[T_0,T_1\right)$}
\label{ConsistencyOfSubsystems}
For readability, we calculate for the first timestep 
%\begin{equation}
\begin{multline}%{split}
\epsilon_{\Delta 1}(t_1) = 
\vec x_{\mathcal D}(t_1)- \vec x_{\Delta, \mathcal D}(t_1)
\\= \Phi^{t_1,t_0} (\left[\vec x_{ \mathcal D,0},\vec x_{ \mathcal I,0}\right])
-\Psi^{t_1,t_0}(\left[\vec x_{\Delta, \mathcal D,0}, \operatorname{Ext}(\vec x_{ \mathcal I,0})\right])
\\=  \underbrace{\Phi^{t_1,t_0} (\left[\vec x_{ \mathcal D,0},\vec x_{ \mathcal I,0}\right]) 
-\Psi^{t_1,t_0}(\left[\vec x_{ \mathcal D,0}, \vec x_{ \mathcal I,0})\right])}_I \\
+ \underbrace{\Psi^{t_1,t_0}(\left[\vec x_{ \mathcal D,0}, \vec x_{ \mathcal I,0}\right])  -\Psi^{t_1,t_0}(\left[\vec x_{\Delta, \mathcal D,0}, \operatorname{Ext}(\vec x_{ \mathcal I,0})\right])}_{II} %\\
\end{multline}%{split}
%\end{equation}
The difference $I$ is of the order consistency of the subsystems method, thus $O(h^{p+1})$.
The difference $II$ 
is
\begin{multline*}
%\left[\vec x_{ \mathcal D,0},\vec x_{ \mathcal I,0}\right] = \left[\vec x_{\Delta, \mathcal %D,0},\operatorname{Ext}(\vec x_{ \mathcal I,0})\right](t_0).
\Psi^{t_1,t_0}(\left[\vec x_{ \mathcal D,0}, \vec x_{ \mathcal I,0}\right])  -\Psi^{t_1,t_0}(\left[\vec x_{\Delta, \mathcal D,0}, \operatorname{Ext}(\vec x_{ \mathcal I,0})\right])
\\= \vec x_{\mathcal D}(t_{0})
+ h\psi\left(t_j,\left[\vec x_{ \mathcal D},  \vec x_{ \mathcal I}\right](t_0), h\right) 
- \left(\vec x_{ \mathcal D}(t_{0}) 
+ h\psi(t_j,\left[\vec x_{  \mathcal D}, \operatorname{Ext}(\vec x_{ \mathcal I})\right](t_0), h)\right)
\end{multline*}
and is estimated using the Lipschitz continuity of $\psi$ and $\vec x_{  \mathcal D, 0}=\vec x_{\Delta, \mathcal D,0}$:
\begin{multline*}
\norm{\cdot} \le hL\norm{\left[\vec x_{ \mathcal D,0},\vec x_{ \mathcal I,0}\right] -  \left[\vec x_{\Delta, \mathcal D,0},\operatorname{Ext}(\vec x_{ \mathcal I,0})\right](t_0)}
\\ \le hL\norm{\vec x_{ \mathcal I,0} -  \operatorname{Ext}(\vec x_{ \mathcal I,0})(t_0)}
\le hLO(h^{P+1})\le LO(h^{P+2}).
\end{multline*}
Remind that capital letters denote exchange parameters: $P$ is the degree of the extrapolation of exchanged data, $H$ the exchange time stepwith.\\
%where $\phi$ is  the taylor series of $\phi$, without its first summand and  divided by $h$. 

%\subsubsection{Consistency of subsystem step inside $\left[T_0,T_1\right)$}

Now consider a subsystem step $j$ inside $[T_0,T_1)$. The grid error is
\begin{multline}\label{splitIntoConsistencyAndPropagation}
\epsilon_{\Delta}(t_{j+1}) 
= \vec x_{\mathcal D}(t_{j+1})- \vec x_{\Delta, \mathcal D}(t_{j+1})\\
= \underbrace{\vec x_{\mathcal D}(t_{j+1})	
-\Psi^{t_{j+1},t_j}(\left[\vec x_{ \mathcal D}, \operatorname{Ext}(\vec x_{ \mathcal I})\right](t_j))}_{=:\vec \epsilon}\\
+ \underbrace{\Psi^{t_{j+1},t_j}(\left[\vec x_{ \mathcal D}, \operatorname{Ext}(\vec x_{ \mathcal I})\right](t_j))
-\Psi^{t_{j+1},t_j}(\left[\vec x_{\Delta, \mathcal D}, \operatorname{Ext}(\vec x_{\Delta, \mathcal I})\right](t_j))}_{=:\vec \epsilon_\text{Prop}}
\end{multline}%{split}
Expression $\vec \epsilon$ resembles the usual local cutoff error made by $\Psi$  in timestep $j$, but has additional contributions from extrapolations,   whereas $\vec \epsilon_\text{Prop}$  describes the propagation of previous errors. It is
\begin{multline}%
%\begin{equation}
%\begin{split}
\vec \epsilon(t_{j+1}) = 
  \vec x_{\mathcal D}(t_{j+1})	
-\Psi^{t_{j+1},t_j}(\left[\vec x_{ \mathcal D}, \operatorname{Ext}(\vec x_{ \mathcal I})\right](t_j))
\\ = \vec x_{\mathcal D}(t_{j})
+ h\psi\left(t_j,\left[\vec x_{ \mathcal D},  \vec x_{ \mathcal I}\right](t_j), h\right) + r(h^{p+1})
\\
- \left(\vec x_{ \mathcal D}(t_{j}) 
+ h\psi(t_j,\left[\vec x_{  \mathcal D}, \operatorname{Ext}(\vec x_{ \mathcal I})\right](t_j), h)\right)%\\
%= \vec x_{\mathcal D}(t_{j}) - \vec x_{\Delta, \mathcal D}(t_{j})+
\\= h\left\{ \psi\left(t_j,\left[\vec x_{ \mathcal D}, \vec x_{ \mathcal I}\right](t_j), h\right)%\right. 
- %\left.
\psi\left(t_j,\left[\vec x_{ \mathcal D}, \operatorname{Ext}(\vec x_{  \mathcal I})\right](t_j), h\right)\right\}\\
+ r(h^{p+1}).
\end{multline}%
%comment
Again using the Lipschitz continuity of $\psi$, we get
\begin{equation}\label{cutoff}
\begin{split}
\norm{\epsilon} \le %hL \norm{\left[\vec x_{ \mathcal D}, \vec x_{  \mathcal I}\right](t_j)	-\left[\vec x_{ \mathcal D}, \operatorname{Ext}(\vec x_{ \Delta, \mathcal I})\right](t_j)} + \norm{r(h^{p+1})}
hL \norm{\left[\vec x_{ \mathcal D}, \vec x_{  \mathcal I}\right](t_j)	-\left[\vec x_{ \mathcal D}, \operatorname{Ext}(\vec x_{  \mathcal I})\right](t_j)} + \norm{r(h^{p+1})}
\\ = hL \norm{	(\vec x_{  \mathcal I}\ - \operatorname{Ext}(\vec x_{  \mathcal I}))(t_j)} + \norm{r(h^{p+1})} 
\\= hLO((jh)^{P+1}) + \norm{r(h^{p+1})}
\\= j^{P+1}hLO(h^{P+1}) + \norm{r(h^{p+1})}
\\\le n^{P+1}O(h^{P+2})+ \norm{r(h^{p+1})} 
%\\ \le h\cdot O((H)^P
\end{split}
\end{equation}
where $n$ is the number of subsystem steps per exchange step. Using exchange stepsize $H=nh$ and provided that the subsystems step size $h$ is bounded by $H$, $h<H$, and that $p>P$,%Remind that capital letters denote exchange parameters.\\
\begin{equation}
 \norm{\epsilon} \le O(H^{P+2})
\end{equation}
formally holds.\todo{right place for discussion?} Unfortunately, in  situations of interest $n$ is bigger than one, thus $h$ in general remains unchanged when $H$ is reduced, so strictly seen
\begin{equation}\label{cutoffWorstCase}
 \norm{\epsilon} \le O(H^{0}) 
\end{equation}
for  $H=nh$,  $n>1$ in typical application. Fortunately,  this still states that $\epsilon$ approaches the error of the subsystems methods, and further, this most common situation will turn out to be covered.\todo{where?}

Consistency as required for theorem \ref{consistencyToConvergence} is $\Phi^{t,t+h}\vec x-\Psi^{t,t+h}\vec x=O(h^{p+1})$, so there is no difference between exact, numerical and extrapolated solution at $t$, so $\operatorname{Ext}(\vec x_{ \Delta, \mathcal I}) = \operatorname{Ext}(\vec x_{  \mathcal I})$.
Using this, expression $\vec \epsilon_{Prop}$ of equation \eqref{splitIntoConsistencyAndPropagation} is calculated as
\begin{multline}\label{propagatedError}%
%\begin{equation}
%\begin{split}
\vec \epsilon_\text{Prop}(t_{j+1}) = \\
  \Psi^{t_{j+1},t_j}(\left[\vec x_{ \mathcal D}, \operatorname{Ext}(\vec x_{ \mathcal I})\right](t_j))
-\Psi^{t_{j+1},t_j}(\left[\vec x_{\Delta, \mathcal D}, \operatorname{Ext}(\vec x_{\Delta, \mathcal I})\right](t_j))
\\ = \vec x_{\mathcal D}(t_{j})
+ h\psi\left(t_j,\left[\vec x_{ \mathcal D},  \operatorname{Ext}(\vec x_{ \mathcal I})\right](t_j), h\right) \\
- \left(\vec x_{\Delta, \mathcal D}(t_{j}) 
+ h\psi(t_j,\left[\vec x_{ \Delta, \mathcal D}, \operatorname{Ext}(\vec x_{ \Delta,\mathcal I})\right](t_j), h)\right)\\
= \vec x_{\mathcal D}(t_{j}) - \vec x_{\Delta, \mathcal D}(t_{j})
\\+ h\left\{ \psi\left(t_j,\left[\vec x_{ \mathcal D}, \operatorname{Ext}(\vec x_{ \mathcal I})\right](t_j), h\right)%\right. 
- %\left.
\psi\left(t_j,\left[\vec x_{\Delta, \mathcal D}, \operatorname{Ext}(\vec x_{  \mathcal I})\right](t_j), h\right)\right\}
\end{multline}%
%\end{split}
%\end{equation}
Thus, using the Lipschitz continuity of $\psi$ again
\begin{multline}%\begin{equation}\begin{split}
\label{recurrenceI}
\norm{\vec\epsilon_\text{Prop}(t_{j+1})} \le \norm{\vec x_{\mathcal D}(t_{j}) - \vec x_{\Delta, \mathcal D}(t_{j})}\\
+hL\norm{\left[\vec x_{ \mathcal D}, \operatorname{Ext}(\vec x_{ \mathcal I})\right](t_j) 
- \left[\vec x_{\Delta, \mathcal D}, \operatorname{Ext}(\vec x_{ \mathcal I})\right](t_j)}
\\ \le\norm{\vec x_{\mathcal D}(t_{j}) - \vec x_{\Delta, \mathcal D}(t_{j})}
+hL_\mathcal D \norm{\vec x_{ \mathcal D}(t_j) 
- \vec x_{\Delta, \mathcal D}(t_j)}
\\ = (1 + hL_{ \mathcal D})\norm{\vec\epsilon_{\Delta}(t_{j})}, %vorher: j-1
%\\+hL_\mathcal I\norm{\operatorname{Ext}(\vec x_{ \mathcal I})(t_j) 
%-  \operatorname{Ext}(\vec x_{\Delta \mathcal I})(t_j)}
\end{multline}%{split}\end{equation}
writing $\vec x_{\mathcal D}(t_{j}) - \vec x_{\Delta, \mathcal D}(t_{j}) = \epsilon_\Delta(t_{j}) $ again. 

\subsubsection{Consistency of a cosimulation step }

First, the situation that exchange  and subsystem stepsize are reduced together, $H_i=ch_i$, is treated, a situation that is occuring  when $H_i=nh_i$ is established, e.g. if $h$ is already reduced by  $H>h$ when $H\longrightarrow 0$  or by reducing subsystems solvers tolerances together with $h$.\\
This situation again splits into the situation when $p=P$ and when $p>P$.
In the first case, 
\begin{equation}\label{cutoffpeqP}
 \norm{\epsilon} \le O(h^{p+1})
\end{equation}
holds, in the second
\begin{equation}\label{cutoffpGreaterP}
 \norm{\epsilon} \le n^{P+1}O(h^{P+2}).
\end{equation}

Equations  \eqref{recurrenceI} and, depending on situation,  \eqref{cutoffWorstCase}, \eqref{cutoffpeqP} or \eqref{cutoffpGreaterP} now give  different  recurrence schemes for the propagated error. 

First, let $H = ch$ be established and $p=P$. In this case \eqref{cutoffpeqP}  %, covering essentially the case  ...
yields the scheme
\begin{equation}\label{recurrencepeqP}
\begin{aligned}
&(i)& \norm{\vec\epsilon_{\Delta}(t_0)} &= & 0 \\
&(ii)&\norm{\vec\epsilon_{\Delta}(t_{j+1})} &\le  &  O(h_j^{p+1}) + (1 + h_jL)\norm{\vec\epsilon_{\Delta}(t_{j})}
\\& & &=&  \tilde C(h^{p+1}) + (1 + h_jL)\norm{\vec\epsilon_{\Delta}(t_{j})}
\end{aligned}
 \end{equation}
 where $\tilde C$ is a constant.\\% if one does not increase the number of subsystem steps per exchange interval.\\
The solution of this recurrence scheme is
\begin{equation}\label{errBoundingFunctionpeqP}
\norm{\vec\epsilon_{\Delta}(t_j)} 
\le 	h^{p}\frac{\tilde C}{L}\left(e^{L(t-t_0)}-1\right) \qquad	\forall t\le T_1
\end{equation}
which is shown in the proof of  \ref{consistencyToConvergence}, see \cite[Th. 4.10]{DeuflhardBornemann94}. The reader is also referred to the solving of the next recurrence scheme, where the arguments are repeated. As
\begin{equation}
\left(  e^{L(t_{n}-t_0)}-1\right) = LH +(LH)^2+O(H^3),
\end{equation}
and $H=ch$
\begin{equation}
\norm{\vec\epsilon_{\Delta}(t_{n})} = O(H^{p+1}) %\text{ for } p=P \text{ and } H=ch
\end{equation}
 is proven for $p\le P$.\\%this error is of order $H^{p+1}$. 

Now, let $H=ch$ be established and $p>P$. The recurrence scheme, using \eqref{cutoffpGreaterP}, now reads
\begin{equation}\label{recurrence}
\begin{aligned}
&(i)& \norm{\vec\epsilon_{\Delta}(t_0)} &= & 0 \\
&(ii)&\norm{\vec\epsilon_{\Delta}(t_{j+1})} &\le & j^{P+1}O(h_j^{P+2}) + (1 + h_jL)\norm{\vec\epsilon_{\Delta}(t_{j})}
\\& & &\le &  n^{P+1} Ch_j^{P+2} + (1 + h_jL)\norm{\vec\epsilon_{\Delta}(t_{j})}.
%\\& & \text{or }&\le &  n^{P+1}O(h^{P+1}) + (1 + HL)\norm{\vec\epsilon_{\Delta}(t_{j})}
%\\& & &\le &  O(H^{P+1}) + (1 + HL)\norm{\vec\epsilon_{\Delta}(t_{j})}
\end{aligned}
 \end{equation}
It is shown that this error is of order $H^{P+1}$ by transferring the arguments from the proof of  \ref{consistencyToConvergence} as given in \cite[Th. 4.10]{DeuflhardBornemann94} to this situation. 
Inductively, 
\begin{equation}\label{errBoundingFunction}
\norm{\vec\epsilon_{\Delta}(t_j)} 
\le 	n^{P+1}h_{\Delta}^{P+1}\frac{C}{L}\left(e^{L(t-t_0)}-1\right) \qquad	\forall t\le T_1
\end{equation}
where $C$ is a constant and $h_\Delta$ the maximum stepwidth in the interval of interest is shown:\\
First, \eqref{errBoundingFunction} holds for $t_0$. Now assume it holds for $t_j$. Then
\begin{align}%
\norm{\vec\epsilon_{\Delta}(t_{j+1})} 
\le& n^{P+1}Ch_j^{P+2} + (1+ h_jL)	n^{P+1} h_\Delta^{P+1}\frac{C}{L}\left(e^{L(t_{j}-t_0)}-1\right)	\\ 
\le&n^{P+1} h_\Delta^{P+1}\frac{C}{L}\left( h_jL+ (1+ h_jL)	\left(e^{L(t_{j}-t_0)}-1\right)\right)
\\ =& n^{P+1}h_\Delta^{P+1}\frac{C}{L}\left(  (1+ h_jL)e^{L(t_{j}-t_0)}-1\right)
\\ \le&n^{P+1} h_\Delta^{P+1}\frac{C}{L}\left(  e^{h_jL}e^{L(t_{j}-t_0)}-1\right)
\\ = &n^{P+1} h_\Delta^{P+1}\frac{C}{L}\left(  e^{L(t_{j+1}-t_0)}-1\right)
\end{align}%
where  $(1+ h_jL)\le e^{h_jL}$  has been applied.
Again using $ \left(  e^{L(t_{n}-t_0)}-1\right) = (LH +(LH)^2+O(H^3))$,
\begin{equation}\label{}
\norm{\vec\epsilon_{\Delta}(t_{n})} = O(n^{P+1}h^{P+1}H) = O(H^{P+2})
\end{equation}
 is proven for $H=ch$  and $p>P$.\\

The preceding two results provide a solution for the recurrence scheme without restriction by relating $h$ and $H$ to each other, using  full equation \eqref{cutoff}: \begin{equation}\label{recurrenceFull}
\begin{aligned}
&(i)& \norm{\vec\epsilon_{\Delta}(t_0)} &= & 0 \\
&(ii)&\norm{\vec\epsilon_{\Delta}(t_{j+1})} &\le  &\tilde Ch_j^{p+1} + n^{P+1} Ch_j^{P+2} + (1 + h_jL)\norm{\vec\epsilon_{\Delta}(t_{j})}.
\end{aligned}
 \end{equation}
Now 
\begin{equation}\label{errBoundingFunction}
\norm{\vec\epsilon_{\Delta}(t_j)} 
\le \frac{1}{L}\left( \tilde Ch_{\Delta}^{p} +	Cn^{P+1}h_{\Delta}^{P+1}\right) \left(e^{L(t-t_0)}-1\right) \qquad	\forall t\le T_1
\end{equation}
is proven inductively. Assume it holds for $t_j$. Then
\begin{align}%
\norm{\vec\epsilon_{\Delta}(t_{j+1})} 
\le
&\tilde Ch_j^{p+1} + n^{P+1}Ch_j^{P+2} 
\\	& + (1+ h_jL) 
 \frac{1}{L}\left( \tilde C h_{\Delta}^{p+1} 
+	Cn^{P+1}h_{\Delta}^{P+1}\right)\left(e^{L(t_{j}-t_0)}-1\right)	\\
\le &
 \tilde C h_j^{p+1} + (1+ h_jL)	\frac{1}{L} \tilde C h_\Delta^{p+1} \left(e^{L(t_{j}-t_0)}-1\right)\label{recI}
\\&+ n^{P+1}Ch_j^{P+2} + (1+ h_jL)	\frac{1}{L}	Cn^{P+1}h_{\Delta}^{P+1}\left(e^{L(t_{j}-t_0)}-1\right)	 \label{recII}
\\ \le
&h_\Delta^{p}\frac{\tilde C}{L}\left(e^{L(t-t_0)}-1\right)
+n^{P+1} h_\Delta^{P+1}\frac{C}{L}\left(  e^{L(t_{j+1}-t_0)}-1\right) 
\\ =
&\left(h_\Delta^{p}\frac{\tilde C}{L}
+n^{P+1} h_\Delta^{P+1}\frac{C}{L}\right)\left(  e^{L(t_{j+1}-t_0)}-1\right). 
\end{align}%
Expression \eqref{recI} equals \eqref{recurrencepeqP} (ii) and so is estimated using \eqref{errBoundingFunctionpeqP},  and expression \eqref{recII} equals \eqref{recurrence} and has been estimated using \eqref{errBoundingFunction}. This finally is, as before,
\begin{equation}\label{consistencyGeneral}
\norm{\vec\epsilon_{\Delta}(t_{j+1})} 
\le \frac{1}{L}\left(h_\Delta^{p}\tilde C
+n^{P+1} h_\Delta^{P+1}C\right)H =O(h^p H) + O(H^{P+2}).
\end{equation}
In the case where the refinement of the exchange time grid does not (yet) influence $h$, this means, after the last refinement $H$ still is a multiple of $h$,
%In this case, equation 
%which turns \eqref{cutoff} into \eqref{cutoffWorstCase}, $\norm{\epsilon} \le O(1) = O(H^0)$, 
%stating that there is an error that does not diminish with $H$.
 %Then 
 the preceding result  applies with $P=0$, giving
 \begin{equation}
\norm{\vec\epsilon_{\Delta}(t_{n})} = O(H).
\end{equation}

\subsubsection{Global convergence}

Consistency of a cosimulation step
 for all subsystems in the cosimulation scheme is given by \eqref{consistencyGeneral}, 
 which is in fact estimating the error of the $i$th subsystems differential states, by reintroducing the subsystems index that could be ommitted in the previous section  $\vec\epsilon_{\Delta,i}(t_{j+1})=\vec x_{\mathcal D,i}(t_j+1)-\vec x_{\Delta, \mathcal D,i}(t_j+1)$,
\begin{equation}
\norm{\vec\epsilon_{\Delta,i}(t_{j+1})} 
\le \frac{1}{L}\left(\tilde C h_\Delta^{p}H
+C H^{P+2} \right).
\end{equation}
So for the consistency error $\norm{\vec\epsilon_{S}(T)}$ of all states of the system,
\begin{equation} \label{consistencySystem}
\norm{\vec\epsilon_{S}(T)} 
=\norm{\begin{pmatrix}\vec\epsilon_{\Delta,1}(t_n)
\\ \vdots
\\ \vec\epsilon_{\Delta,m}(t_n)
\end{pmatrix}}
\le \frac{1}{L}\left(\tilde{\tilde C} h_\Delta^{p}H
+\overline C H^{P+2} \right),
\end{equation}
so
the cosimulation method applied to overall system $S$ is \emph{consistent} according to definition \eqref{defConsistency} with respect to methods stepwith $H$. 

Using this, one proceeds in the straightforward way: The recurrence inequality now reads
\begin{equation}
\norm{\vec\epsilon_{\Delta,S}(T_{j+1})}\le \frac{1}{L}\left(\tilde{\tilde C} h_\Delta^{p}H +\overline C H^{P+2} \right)+ (1+HL)\vec\epsilon_{\Delta,S}(T_{j})
\end{equation}
where the index $S$ indicates that the error of whole system $S$ is considered. Calculation is performed in the appendix and is completely analogous to equations \eqref{splitIntoConsistencyAndPropagation}-\eqref{recurrenceI}.\\
The recurrence scheme is solved again analogously to the proof of \eqref{errBoundingFunction}:
\begin{equation}\label{convErrBoundingFunction}
\norm{\vec\epsilon_{\Delta,S}(T_{j+1})} 
\le \frac{1}{L}\left(\tilde{\tilde C}h_\Delta^{p}
+CH^{P+1}\right)\left(  e^{L(T_{j+1}-T_0)}-1\right)
\end{equation}
is proven inductively. Assume it holds for $t_j$. Then
\begin{align}%
\norm{\vec\epsilon_{\Delta,S}(T_{j+1})} 
\le
& \frac{1}{L}\left(\tilde{\tilde C} h_\Delta^{p}H +\overline C H^{P+2} \right) 
\\	& + (1+ HL) 
 \frac{1}{L}\left(\tilde{\tilde C}h_\Delta^{p}
+CH^{P+1}\right)\left(  e^{L(T_{j}-T_0)}-1\right)
\\ \le &
 \quad \tilde{\tilde C} h_\Delta^{p}H + (1+ HL)	\frac{1}{L} \tilde{\tilde C} h_\Delta^{p}\left(e^{L(T_{j}-T_0)}-1\right)\label{recIGlobal}
\\&+ \overline C H^{P+2} + (1+ HL)	\frac{1}{L}	\overline C H^{P+1}\left(e^{L(T_{j}-T_0)}-1\right)	 \label{recII}
\\ \le&
\quad \frac{1}{L}\tilde{\tilde C} h_\Delta^{p}\left(HL + (1+ HL) \right) \left(e^{L(t_{j}-T_0)}-1\right)
\\ &+ \frac{1}{L}	\overline C H^{P+1}  \left(HL + (1+ HL) \right)   \left(e^{L(t_{j}-T_0)}-1\right)
\\ \le
&\frac{1}{L}\left( \tilde{\tilde C} h_\Delta^{p}    + \overline C H^{P+1}  \right)\left(  e^{L(T_{j+1}-T_0)}-1\right). 
\end{align}%

By this result, refining the subsystems step size alone does not ensure convergence, and refinement of exchange step sizes alone while it does not bound $h$ has  an offset in error:
\begin{equation}\label{convergenceCosim}
\norm{\vec\epsilon_{\Delta}(T)} = 
\begin{cases}
O(H^{P+1}) & \text{ if } ch=H \text{ and } p>P\\
O(H^{p}) &\text{ if } ch=H \text{ and } p\le P\\\
O(1) & \text{ if $h$ and $H$ are independent} 
\end{cases}
\text{ in }[T_0, T_{end}).
\end{equation}
Thus unique rates might be less clearly observable as usual in practice due to the presence of  contributions of different order and of the two variables $H$ and $h$.
See section \ref{numResultsPure} for examples.
%It may for example happen that the diminishing of the $O(H^{P+1})$ -part is observed during numerical experiments rather than the offset $O(1)$  if it is small and the second contribution dominates observation.

The above clear cases are consistent with the conclusions one can draw from  
 theorem \ref{consistencyToConvergence} using consistency orders 
\begin{equation}
\norm{\vec\epsilon_{\Delta}(t_n)} = 
\begin{cases}
O(H^{P+2}) & \text{ if } ch=H \text{ and } p>P\\
O(H^{p+1}) &\text{ if } ch=H \text{ and } p\le P\\\
O(H^1) & \text{ if $h$ and $H$ are independent.} 
\end{cases}
\text{ for } t_n \in [T_0, T_{1}).
\end{equation}
Moreover, this is  consistent with estimates given in \cite[Example 3]{ArnoldClaussSchierz2013}. Finally, the result is given as a 
\begin{theorem}\label{th:convergence}
Let S be a set of ODE which is split into disjoint subsystems $S_k$ of the shape 
\begin{equation}
\dot{\vec x}_{\mathcal D_k} = f(t,\left[{\vec x}_{\mathcal D_k},\vec x_{\mathcal I_k}\right] ),
\end{equation}
the $\mathcal D_k$ and $\mathcal I_k$ denoting index sets.
Let ${T_i}$ be a time grid with width $H$, and
let the inputs of all subsystems $\vec x_{\mathcal I_k}$ be extrapolated at $T_i$ with polynomial order $P$ and then be solved with an one-step method of order  $p$ and maximal stepwidth $h\le H$. Then for the error of the numerical solution the estimate 
\begin{equation}\label{convErrBoundingFunction}
\norm{\vec\epsilon_{\Delta,S}(T_{j+1})} 
\le \frac{1}{L}\left(\tilde{\tilde C}h_\Delta^{p}
+CH^{P+1}\right)\left(  e^{L(T_{j+1}-T_0)}-1\right)
\end{equation}
holds.
\end{theorem}

\paragraph{Remark} For extrapolation of input data to a certain order, exact values and derivatives of it are required. Those
values and derivatives may depend on up-to-date input. If dependencies on input are circular, it thus may be impossible to get exact derivatives without some iterating procedure for solving the coupling equations \eqref{coupling1} and \eqref{coupling2}, as e.g. presented in \cite{KueblerSchiehlen2000}.  When estimating the consistency error (eq. \eqref{splitIntoConsistencyAndPropagation}), thus order loss may occur.

\subsection{Stability }\label{Stability}
Aim of this section is to examine the stability of the overall system.

 According to the established notions  of A- and B-stability in standard ODE theory, a method is called A- (resp. B-) stable if it conserves the stability properties of the ODE.\\
First, A-stability is a property of performance of methods on the scalar problem $\dot x = \lambda x$ and as such of no use for examination of coupling schemes. We
instead examine \emph{linear stability},
% examine the methods stability using a \emph{vector A-stability} 
which shall be the conservation of linear vector valued problems
\begin{equation}
\dot{\vec x} = \ten B \vec x%, \ten A = \ten Q^T \ten D \ten Q
\end{equation}
 stability by the method in question.

Consider a system $S$ governed by $\dot{\vec x} = \ten B \vec x$, which is split into  subsystems $S_i$ with variables $x_i$, $i=1,2$ %- let in this section superscript indices denote the system in which they are differential variable, as we need more indices for steps - 
 governed by 
\begin{align}
\dot x_1 = & B_{11}x_1 +B_{12}x_2 =: f_1(x_1,x_2)\\
\dot x_2 = & B_{22}x_2 +B_{21}x_1.
\end{align}
The last summands in this setting are inputs and in  previously used notation one would write
$(\vec x_{\mathcal D},\vec x_{\mathcal I})=(x_1, x_2)$. Cosimulation  method means
\begin{align}
\dot x_1 = & B_{11}x_1 +B_{12}\operatorname{Ext}(x_2) \label{eq:methodInducedLin} \\
\dot x_2 = & B_{22}x_2 +B_{21}\operatorname{Ext}(x_1).
\end{align}
The method is not stable if it does not preserve the stability of the ODE, i.e. if there is a matrix $\ten B$ with $\rho(\ten B)<1$ but 
\begin{equation}
\Psi^H\vec x_i >  \vec x_i.
\end{equation}
Choose $B_{11}=0$ and $B_{12}\neq 0$ but subject to $\rho(\ten B)<1$. Then the solution of the method induced ODE \eqref{eq:methodInducedLin} at $t_{k}\in [T_j,T_{j+1}]$ is 
\begin{equation}
x_1(t_{j}) =  0 +B_{12}\int_{T_j}^{t_{k}}\operatorname{Ext}(x_2)dt.
\end{equation}
For order $n$ extrapolation the integral  is a polynomial of degree $n+1$,  for constant extrapolation that is $x_2t$ - the solution of the method induced ODE such is not stable for any extrapolation, so the numerical solution cannot be stable for any subsystems integration scheme.\\
We state  that from B-stability linear  stability follows:  Consider above linear ODE,  which shall be stable, and thus $\operatorname{Re}(\rho(\ten B))\le 0$. From these properties $\scalar{\ten B x}{x}\le 0$, which is  $\scalar{\ten B(y_1-y_2)}{(y_1-y_2)}\le 0$ and so the dissipativity. So the discrete solution of a  B-stable method applied to a stable linear ODE will be stable.
As linear stability is not given for explicite cosimulation schemes, explicite cosimulation is neither B-stable. \\
These results could be expected, if one sees cosimulation methods as mixed explicit-implicit methods.\\
%\todo{Deufl 6.13.Method A-Stab.=> $\psi x$ asympt. stable for IVP $\dot x = Ax$}

%One has to show that 
%\begin{equation}
%\left|(1+SB_{11})^n +SB_{12}\sum_{i=0}^{n-1}(1+S_{B11})^i
%\right| > 1
%\end{equation}
%for a $B_{11} h $ with negative real part. 
%If that,
%\begin{equation}
%\left|SB_{12}\sum_{i=0}^{n-1}(1+SB_{11})^i\right|
% > 1 - \left|(1+SB_{11})^n\right| .
%\end{equation}
%must hold. Now Subsystems Runge-Kutta method is assumed to be A-stable:
%\todo{Bullshit:}If that would not be given, the subsystems method could not even preserve stability of solution inside $[T_i, T_{i-1})$, so this is necessary for stability of the cosimulation. So for an arbitary but fixed value \todo{lambda?} $B_{11} h,$
%\begin{equation}
%\norm{1+SB_{11}}<1 \quad \forall \quad B_{11}h:\quad \operatorname{Re}(B_{11}h)< 0.
%\end{equation}
%can be assumed.\\
%Thus $1 - (1+SB_{11})^n > 0$, and 
%\begin{equation}
%SB_{12}\sum_{i=0}^{n-1}(1+S_{B11})^i > c
% > 0
%\end{equation}
%with some c, and using $(1-p^n)/(1-p^n)=\sum_{i=0}^{n-1}p^i$, it becomes
%\begin{equation}
%SB_{12}\frac{1-(1+SB_{11})^n}{1-(1+SB_{11})} > c
%\end{equation}

%\pagebreak

%% file: convergenceOfBalanceCorrected.tex
\section{Convergence of balance correction methods}\label{sec:convergenceBC}

Balance correction methods  seek to establish conservation of an input quantity $\vec x_{ \mathcal I}$ by adding its extrapolation error to the input at a later timestep, when it can be calculated. Denoting 
\begin{equation}
\operatorname{bal}_j:= g(t, t_j,h)\epsilon_{B, j}
\end{equation}
where 
\begin{equation}
\epsilon_{B, j} = \int\limits_{t_{j-1}}^{t_j}(\vec x_{\Delta,\mathcal I } - \operatorname{Ext}(\vec x_{\Delta ,\mathcal I}))(\tau) \,d\tau
\end{equation}
and
$g(t, t_i,h)$ is a smooth function with a support beginning at $t_i$, extending to 1-4 data exchange intervals and fulfilling $\int g(t, t_i,h)\,dt = 1$, the general one-step balance correction method reads 
 \begin{equation}\begin{split}
 \Psi^{t_{j+1},t_j}_\text{corr}(\left[\vec x_{ \mathcal D}, \operatorname{Ext}(\vec x_{ \mathcal I})\right](t_j)) 
 \\= \left[\vec x_{ \mathcal D}, \operatorname{Ext}(\vec x_{ \mathcal I})\right](t_j)
 + h\psi\left(t_j,\left[\vec x_{ \mathcal D}, \operatorname{Ext}(\vec x_{ \mathcal I}) + \operatorname{bal}_j\right],h \right).
\end{split}
\end{equation}

Our examination now follows exactly the arguments outlined  in the introduction of \ref{convergence} and
treats the extrapolation error contribution in the method as an input error and shows that this input error does not spoil validity of any of the arguments used.

\subsection{Consistency of subsystem step inside $\left[T_{j},T_{j+1}\right)$}

One may  split $\epsilon_{\Delta }(t_j)$   for a balance correction method according to \eqref{splitIntoConsistencyAndPropagation} and further, such that 
\begin{multline}%
 %\begin{equation}\begin{split}
\epsilon_{\Delta }(t_j) 
=\vec x_{\mathcal D}(t_{j+1})  - \Psi^{t_{j+1},t_j}_\text{corr}(\left[\vec x_{\Delta, \mathcal D}, \operatorname{Ext}(\vec x_{\Delta, \mathcal I})\right](t_j))
\\=\underbrace{\vec x_{\mathcal D}(t_{j+1})  
-\Psi^{t_{j+1},t_j}(\left[\vec x_{ \mathcal D}, \operatorname{Ext}(\vec x_{ \mathcal I})\right](t_j))}_{\vec \epsilon} 
\\ + \underbrace{\Psi^{t_{j+1},t_j}(\left[\vec x_{ \mathcal D}, \operatorname{Ext}(\vec x_{ \mathcal I})\right](t_j))
-\Psi^{t_{j+1},t_j}(\left[\vec x_{ \Delta\mathcal D}, \operatorname{Ext}(\vec x_{ \Delta,\mathcal I})\right](t_j))}_{\vec \epsilon_\text{Prop}}
\\+\underbrace{\Psi^{t_{j+1},t_j}(\left[\vec x_{ \Delta\mathcal D}, \operatorname{Ext}(\vec x_{ \Delta,\mathcal I})\right](t_j))
- \Psi^{t_{j+1},t_j}_\text{corr}\left(\left[\vec x_{ \Delta, \mathcal D}, \operatorname{Ext}(\vec x_{ \Delta,\mathcal I})\right](t_j)\right) }_{\vec \epsilon_\text{bal}}.
%\end{split}
\end{multline}
\begin{itemize}
\item The first difference  $\epsilon$ is the same as in equation \eqref{splitIntoConsistencyAndPropagation}, thus named like it and thus  yields
\begin{multline}%
%\begin{equation}
%\begin{split}
\vec \epsilon(t_{j+1}) = 
  \vec x_{\mathcal D}(t_{j+1})	
-\Psi^{t_{j+1},t_j}(\left[\vec x_{ \mathcal D}, \operatorname{Ext}(\vec x_{ \Delta,\mathcal I})\right](t_j))
\\= h\left\{ \psi\left(t_j,\left[\vec x_{ \mathcal D}, \vec x_{ \mathcal I}\right](t_j), h\right)%\right. 
- %\left.
\psi\left(t_j,\left[\vec x_{ \mathcal D}, \operatorname{Ext}(\vec x_{ \Delta, \mathcal I})\right](t_j), h\right)\right\}\\
+ r(h^{p+1}),
\end{multline}%
which then is estimated analogously to \eqref{cutoff} as 
\begin{equation}
\norm{\epsilon} \le n^{P+1}O(h^{P+2})+ \norm{r(h^{p+1})}.
\end{equation}
\item Expression $\vec \epsilon_\text{Prop}$   is the propagated error again and is  treated as in equation \eqref{propagatedError}  using the Lipschitz continuity. The fact that  consistency as required for theorem \ref{consistencyToConvergence} is $\Phi^{t,t+h}\vec x-\Psi^{t,t+h}\vec x=O(h^{p+1})$ justifies the use of  $\operatorname{Ext}(\vec x_{  \mathcal I}) = \operatorname{Ext}(\vec x_{ \Delta, \mathcal I})$ here again. This results in  equation \eqref{recurrenceI} again:
\begin{multline*}%
\norm{\vec\epsilon_\text{Prop}(t_{j+1})}
=\norm{\Psi^{t_{j+1},t_j}(\left[\vec x_{ \mathcal D}, \operatorname{Ext}(\vec x_{ \mathcal I})\right](t_j))
-\Psi^{t_{j+1},t_j}(\left[\vec x_{ \Delta\mathcal D}, \operatorname{Ext}(\vec x_{ \Delta,\mathcal I})\right](t_j))}
 \\ = (1 + hL_{ \mathcal D})\norm{\vec\epsilon_{\Delta}(t_{j})}\qquad \forall j<n.
\end{multline*}%

\item Finally, $\vec \epsilon_\text{bal}$ becomes 
\begin{multline}
\vec \epsilon_\text{bal}=\Psi^{t_{j+1},t_j}(\left[\vec x_{ \Delta,\mathcal D}, \operatorname{Ext}(\vec x_{ \Delta,\mathcal I})\right](t_j))
- \Psi^{t_{j+1},t_j}_\text{corr}\left(\left[\vec x_{ \Delta, \mathcal D}, \operatorname{Ext}(\vec x_{ \Delta,\mathcal I})\right](t_j)\right)
\\=\left[\vec x_{\Delta, \mathcal D}, \operatorname{Ext}(\vec x_{\Delta, \mathcal I})\right](t_j)
 + h\psi\left(t_j,\left[\vec x_{ \Delta,\mathcal D}, \operatorname{Ext}(\vec x_{ \Delta,\mathcal I})\right],h \right)
\\- 
\left\{\left[\vec x_{\Delta, \mathcal D}, \operatorname{Ext}(\vec x_{\Delta, \mathcal I})\right](t_j)
 + h\psi\left(t_j,\left[\vec x_{ \Delta,\mathcal D}, \operatorname{Ext}(\vec x_{ \Delta,\mathcal I}) + \operatorname{bal}_j\right],h \right)\right\}
\\ \le h L\norm{\left[\vec x_{ \Delta,\mathcal D}, \operatorname{Ext}(\vec x_{ \Delta,\mathcal I})\right] 
- \left[\vec x_{ \Delta,\mathcal D}, \operatorname{Ext}(\vec x_{ \Delta,\mathcal I}) + \operatorname{bal}_j\right]}
\\ = h L\norm{\left[0, \operatorname{bal}_j\right]}.
\end{multline}
Mind that $\operatorname{bal}_j$ is part of the method and has to be considered during  consistency analysis even if otherwise  $\operatorname{Ext}(\vec x_{  \mathcal I}) = \operatorname{Ext}(\vec x_{ \Delta, \mathcal I})$ is assumed on $[T_0,T_1]$.
By the Cauchy-Schwarz inequality
\begin{multline}
\norm{g(t,t_j,h) \int\limits_{t_{j-1}}^{t_j}(\vec x_{\Delta,\mathcal I } - \operatorname{Ext}(\vec x_{\Delta ,\mathcal I}))(\tau) \,d\tau}
\\ \le 1 h \max_{t\in [t_{j-1}, t_j)}(\left|\vec x_{\Delta,\mathcal I } - \operatorname{Ext}(\vec x_{\Delta ,\mathcal I})\right|(t))
\\ = hO(jh^{P+1}) \le  n^{P+1}O(h^{P+2}).
\end{multline}
\end{itemize}
Together,
\begin{equation}
\vec \epsilon_\text{bal}\le h L n^{P+1}O(h^{P+2})= n^{P+1}O(h^{P+3}).
\end{equation}
If $H$ while $h$ remains constant, this is $H^{P+1}$, if $h$ is reduced while $n$ constant, this is $h^{P+3}$, so higher order compared to $\vec \epsilon$ in any case.
%As this contribution is of order $O(h^{P+2})$, again $P+2> p+1$ assumed, the 
The recurrence scheme can be set up the same way as in section \ref{convergence},
\begin{equation}
\begin{aligned}
&(i)& \norm{\vec\epsilon_{\Delta}(t_0)} &= & 0 \\
&(ii)&\norm{\vec\epsilon_{\Delta}(t_{j+1})} &\le & n^{P+1}O(h^{P+2}) + (1 + h_jL)\norm{\vec\epsilon_{\Delta}(t_{j})}
\\& & &\le & \tilde Ch_j^{p+1} + n^{P+1} Ch_j^{P+2} + (1 + h_jL)\norm{\vec\epsilon_{\Delta}(t_{j})}
%before:  O(H^{P+2}) + (1 + HL)\norm{\vec\epsilon_{\Delta}(t_{j})},
\end{aligned}
 \end{equation}
 with the same result \eqref{errBoundingFunction} 
\begin{equation}
\norm{\vec\epsilon_{\Delta}(t_j)} 
\le \frac{1}{L}\left( \tilde Ch_{\Delta}^{p+1} +	Cn^{P+1}h_{\Delta}^{P+1}\right) \left(e^{L(t-t_0)}-1\right) \qquad	\forall t\le T_1
\end{equation}
 and same implications \eqref{consistencyGeneral} for consistency and  \eqref{convergenceCosim} for global convergence:
 \begin{equation}\label{convergenceBC}
\norm{\vec\epsilon_{\Delta}(t)} = 
\begin{cases}
O(H^{P+1}) & \text{ if } ch=H \text{ and } p>P\\
O(H^{p}) &\text{ if } ch=H \text{ and } p\le P\\\
O(1) & \text{ if $h$ and $H$ are independent} 
\end{cases}
\text{ in }[T_0, T_{end}).
\end{equation}
Finally, the result is formulated as 
\begin{theorem}\label{th:convergenceBC}
Let $S$, $S_k$, ${\vec x}_{\mathcal D_k}$, $\vec x_{\mathcal I_k}$ and ${T_i}$ be as in theorem \ref{th:convergence}, but balance correction contributions be added to the extrapolated variables.
Then the estimate from  theorem \ref{th:convergence} still holds.
\end{theorem}

 This results can be interpreted in the following way: If we regard the balance correction as an errorous disturbation, it still is too small to disturb convergence.\\
 It should be possible to show that for a sufficiently small macro time step size, the error of a cosimulation method with balance correction is smaller than one without balance correction.

%% file: numResultsEnergyI.tex
%nur in convergence eingebunden! in energyBalancing klonbasiertes Derivat.
%Name mislading
\subsection{Numerical results}\label{numResultsPure}
\subsubsection{Linear Problem for convergence examination} 
For examination of convergence, a two-component linear system split into two subsystems  of one variable each is solved, which written in FMI style reads:\\[11pt]
\begin{tabular}{p{0.22\textwidth} | p{0.22\textwidth}} 
\multicolumn{2}{ c }{General scheme}\tabularnewline
\center{$S_1$} & \center{$S_2$}\tabularnewline
%\begin{equation}
\multicolumn{2}{ c }{System States}\tabularnewline
\center{$\vec x_1  $} 		&\center{$\vec x_2 $} \tabularnewline
\multicolumn{2}{ c }{Outputs}\tabularnewline
\center{$\vec y_1  $} 		&\center{$\vec y_2 $} \tabularnewline
\multicolumn{2}{ c }{Inputs}\\
\center{$\vec u_1=\vec y_2$}	& \center{$\vec u_2 = \vec y_1$}\tabularnewline
\multicolumn{2}{ c }{Equations}\tabularnewline
\center{$\dot {\vec x_1} = \vec f_1(\vec x_1, \operatorname{Ext}(\vec u_1))	$} & \center{$\dot{ \vec x_2} = \vec f_2(\vec x_2, \operatorname{Ext}(\vec u_2))$}	\\
\end{tabular}
\begin{tabular}{p{0.24\textwidth} | p{0.24\textwidth}} \multicolumn{2}{ c }{$\dot{\vec x} = \ten A \vec x$}\tabularnewline
\center{$S_1$} & \center{$S_2$}\tabularnewline
%\begin{equation}
\multicolumn{2}{ c }{System States}\tabularnewline
\center{ $ x_1  $ } 		&\center{$ x_2 $} \tabularnewline
\multicolumn{2}{ c }{Outputs}\tabularnewline
\center{$ y_1:= x_{1}  $} 		&\center{$ y_2:= x_{2} $} \tabularnewline
\multicolumn{2}{ c }{Inputs}\\
\center{$  u_1:=  y_2$}	& \center{$ u_2 :=  y_1$}\tabularnewline
\multicolumn{2}{ c }{Equations}\tabularnewline
\center{$\dot { x_1} = a_{1,1}x_1 + a_{1,2}\operatorname{Ext}( u_1) $} & \center{$\dot{  x_2} = a_{2,2}x_2 + a_{2,1} \operatorname{Ext}(u_2) $}	\\
\end{tabular}\\[12pt]
 Efficiently, the coupled equations are:
\begin{equation}\label{LinProb}
\dot{\vec x} = \ten A \vec x \quad \Rightarrow \quad 
\begin{cases}
\dot x_1 = a_{1,1}x_1 + a_{1,2}\operatorname{Ext}(x_2) 
\\ \dot x_2 = a_{2,2}x_2 + a_{2,1}\operatorname{Ext}(x_1). 
\end{cases}
\end{equation}
The matrix entries are chosen such that 
\begin{itemize}
\item an unidirectional dependency on input is given: $a_{11},a_{21},a_{22}\neq 0$, $a_{12}=0$,
\item a mutual dependency is given: $a_{12},a_{21}\neq 0$, $a_{11}=a_{22}=0$.
\end{itemize}
In the unidirectional flow of data case, one expects the convergence rates from \eqref{convErrBoundingFunction}  to become observable. 
In the example where the data flows in both directions, the linear extrapolation of input data case is an example for a problem where the coupling equations \eqref{coupling1}-\eqref{coupling2} are not exactly fulfilled when calculating the $\vek u_i$ and where one would have to apply a method to make outputs consistent as discussed in section
\ref{sec_consistentInput};
it might illuminate the effects of output dependency on past-timestep input data: In this case, the output row in above scheme reads\\
\begin{center}
\begin{tabular}{p{0.27\textwidth} | p{0.27\textwidth}} 
\multicolumn{2}{ c }{$\vdots$}\tabularnewline
\multicolumn{2}{ c }{Outputs}\tabularnewline
\center{$ \vec y_1:= [x_{1},  \dot x_{1}]$} 		&\center{$ \vec y_2:= [x_{2} ,\dot x_{2}]$} \tabularnewline
\multicolumn{2}{ c }{Inputs}\\
\center{$\vec u_1:=\vec y_2$}	& \center{$\vec u_2 := \vec y_1$}\tabularnewline
\multicolumn{2}{ c }{$\vdots$}\tabularnewline
\end{tabular}\end{center}%\\[12pt]
and as
\begin{equation}
\begin{split}
\dot x_i(t) = a_{i,i}x_i(t) + a_{i,j}\operatorname{Ext}(x_j)(t),\\  \operatorname{Ext}(x_j)(t)=x_j(T_N)+\dot x_j(T_N)(t-T_N)
\\=x_j(T_N) + (a_{j,j}x_j(T_N) + a_{j,i}\operatorname{Ext}(x_i)(T_N))(t-T_N),
\end{split}
\end{equation}
and either $\operatorname{Ext}(x_i)(T_N)$ or $\operatorname{Ext}(x_j)(T_N)$ in the corresponding equation of the subsystem $j$  has to be the extrapolant from $(T_{N-1},T_N]$,  one expects a bigger error for this problem. On the other hand, with $\operatorname{Ext}^{(T_{N-1},T_N]}(x_i)(T_N)$ denoting the extrapolant for the interval given in superscript,
\begin{equation}
\operatorname{Ext}^{(T_{N},T_N+1]}(x_i)(t)
=\operatorname{Ext}^{(T_{N-1},T_N]}(x_i)(t)+(t-T_{N-1})\left(\dot x_i(T_{N-1})
+\partial_ux_i\dot u_i  \right)
\end{equation}
and so an order loss is not to be expected.\\
\begin{figure}%[h!tb]
\includegraphics[ width=0.8\textwidth]{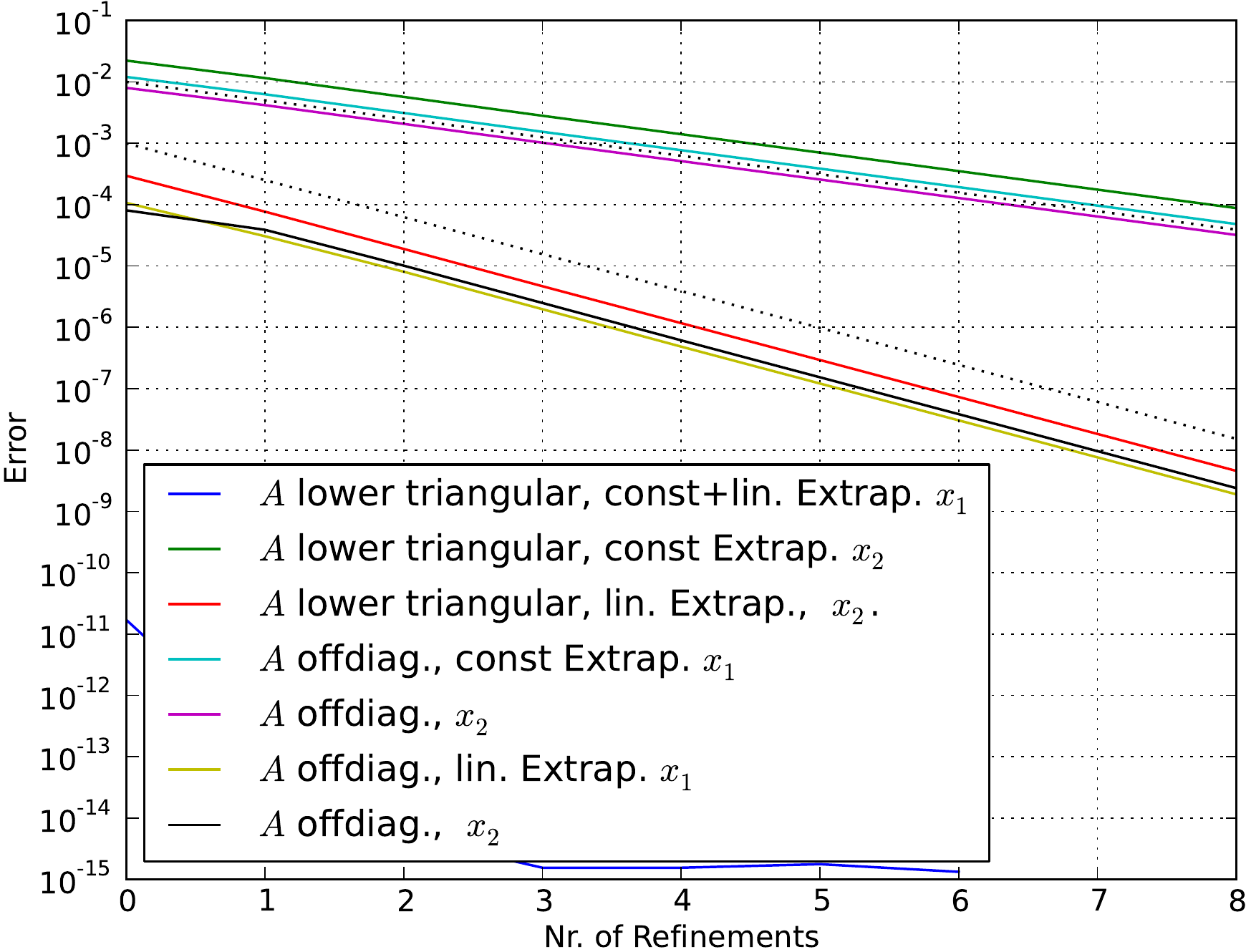}
\caption{\label{CosimConvergencePlot} Simulation of the system \eqref{LinProb} with matrix $A\in \mathbb R^2$ realizing one-directional coupling (lower triangular shape) resp. mutual coupling (only offdiagonal entries are nozero) in the cosimulation scheme with constant and linear extrapolation, varying the exchange step size $H$.  }
\end{figure}
Figure \ref{CosimConvergencePlot} shows the convergence result for the four situations.   To show the predictions made by \eqref{convErrBoundingFunction} in terms of $H$, it is necessary that the subsystems methods contribution $\tilde{\tilde C} h_\Delta^{p}$ is of higher order than the extrapolation and that the method used is a one-step method, which  will be discussed in Section \ref{pitfallMultistep}. Thus \emph{dopri5}, an explicite Runge-Kutta method, was chosen, using the built-in stepsize control with default absolute tolerance $10^{-12}$.\\
The figure shows that convergence is of order 1 for constant extrapolation and of order 2 for linear extrapolation, as predicted by \eqref{convErrBoundingFunction}. As discussed, there is no order loss for linear extrapolation and circular dependency of inputs, but not even an higher error, in spite of the negative effects that should occur.\\ 
The error of the first component of the lower triangular, thus unidirectionally coupled system is very low as it has no extrapolation contribution, thus indicating that the error made by $dopri$ is low enough to allow for judgement of the effect of extrapolation error.

\subsubsection{Spring-Mass system for convergence and stability} 
\label{springMassConvergence}
To numerically examine the stability of the method, we chose the linear spring-mass oscillator as discussed in  \cite{ScharffMoshagen2017}  and given by 
\begin{equation} \label{springMass}
\dot{ \vec{ x}}
= \ten A\vec{x} 
= \begin{pmatrix}
0 & 1\\
-\frac{c}{m} & \left(-\frac{d}{m}\right)
 \end{pmatrix} \vec x, 
 \qquad \vec{x} = 
 \begin{pmatrix}
x\\
\dot x
 \end{pmatrix}
\end{equation}
with mass $m$, spring constant $c$, and in which the damping constant $d$ shall vanish, for being simple while still showing stability problems. It is indeed a physical interpretation of the above system with offdiagonal matrix.\\ 
%\begin{tabular}{p{0.24\textwidth} | p{0.24\textwidth}} 
%\multicolumn{2}{ c }{General scheme}\tabularnewline
%\center{$S_1$} & \center{$S_2$}\tabularnewline
%\multicolumn{2}{ c }{System States}\tabularnewline
%\center{$\vec x_1  $} 		&\center{$\vec x_2 $} \tabularnewline
%\multicolumn{2}{ c }{Outputs}\tabularnewline
%\center{$\vec y_1  $} 		&\center{$\vec y_2 $} \tabularnewline
%\multicolumn{2}{ c }{Inputs}\\
%\center{$\vec u_1=\vec y_2$}	& \center{$\vec u_2 = \vec y_1$}\tabularnewline
%\multicolumn{2}{ c }{Equations}\tabularnewline
%\center{$\dot {\vec x_1} = \vec f_1(\vec x_1, \operatorname{Ext}(\vec u_1))	$} & \center{$\dot{ \vec x_2} = \vec f_2(\vec x_2, \operatorname{Ext}(\vec u_2))$}	\\
%\end{tabular}
%\todo{Dirk: vielleicht so sehen - gleiche Rechnung, aber so haben beide subsysteme einen state und eine Zeile des ODE-Systems!}
\begin{table}[h!tb]
\begin{tabular}{c|c} %m{0.25\textwidth} | m{0.25\textwidth}} %\multicolumn{2}{ c }{Spring-Mass}\tabularnewline
Spring  & Mass\tabularnewline
%\begin{equation}
\multicolumn{2}{ c }{System States}\tabularnewline
%bullshit: \parbox[t]{0.25\textwidth}{ 
 $ x_1 := s =  x $ 		&  $ x_2 := v = \dot x $ \tabularnewline
\multicolumn{2}{ c }{Outputs}\tabularnewline
$ y_1:=F =-cx  $		&$ y_2:= v = \dot x $ \tabularnewline
\multicolumn{2}{ c }{Inputs}\\
$  u_1:=  y_2$	& $ u_2 :=  y_1$\tabularnewline
\multicolumn{2}{ c }{Equations}\tabularnewline
\parbox[t]{0.24\textwidth}{ \begin{equation*}\begin{split}\dot { x_1} = \operatorname{Ext}( u_1) = v\end{split}\end{equation*}	} 
& \parbox[t]{0.24\textwidth}{\begin{equation*}\begin{split}\dot{  x_2} = -\frac{1}{m} \operatorname{Ext}(u_2) \\= -\frac{F}{m}\end{split}\end{equation*}}	\\
\end{tabular}
\begin{tabular}{c|c}%p{0.24\textwidth} | p{0.24\textwidth}} %\multicolumn{2}{ c }{Spring-Mass}\tabularnewline
Spring & Mass\tabularnewline
%\begin{equation}
\multicolumn{2}{ c }{System States}\tabularnewline
 	$\vdots$	&  $\vdots$\tabularnewline
\multicolumn{2}{ c }{Outputs}\tabularnewline
\parbox[t]{0.24\textwidth}{$ \vec y_1:=(f,\dot f)$\\ $=(-cx, -cv) $ } 		&\parbox[t]{0.24\textwidth}{ $\vec y_2:=( v,a)$ \\$= (\dot x, \dot f/m )$} \tabularnewline
\multicolumn{2}{ c }{Inputs}\\
$  u_1:=  y_2$	& $ u_2 :=  y_1$\tabularnewline
\multicolumn{2}{ c }{Equations}\tabularnewline
 $\vdots$&$\vdots$	\tabularnewline
 &\tabularnewline
 &
\end{tabular}%\\[12pt]
\caption{\label{CosimSchemesSpringMass}Cosimulation Schemes for the spring-mass system, left constant, right linear extrapolation}
\end{table}
This system was treated  in the cosimulation scheme \ref{CosimSchemesSpringMass}.
%\\[12pt]
Output of the spring is the force $F = -cx$, %m\ddot x$, 
that of the mass is the velocity $v = \dot x$. As ODE solver on subsystems, any solver that does not dominate the convergence and stability behavior of the cosimulation scheme could be used. The plots \ref{ConstCosimConvergencePlot} and \ref{LinCosimConvergencePlot} show simulations done with \emph{vode} and \emph{zvode} from the \emph{numpy} Python numerics library, which both implement implicit Adams method if problem is nonstiff and BDF if it is and behave stable due to their step size adjustment. The convergence plot \ref{ConvergencePlot}  was made with \emph{dopri5} (see Section \ref{pitfallMultistep}). \\
\begin{figure}%[h!tb]
\center{\includegraphics[ width=0.75\textwidth]{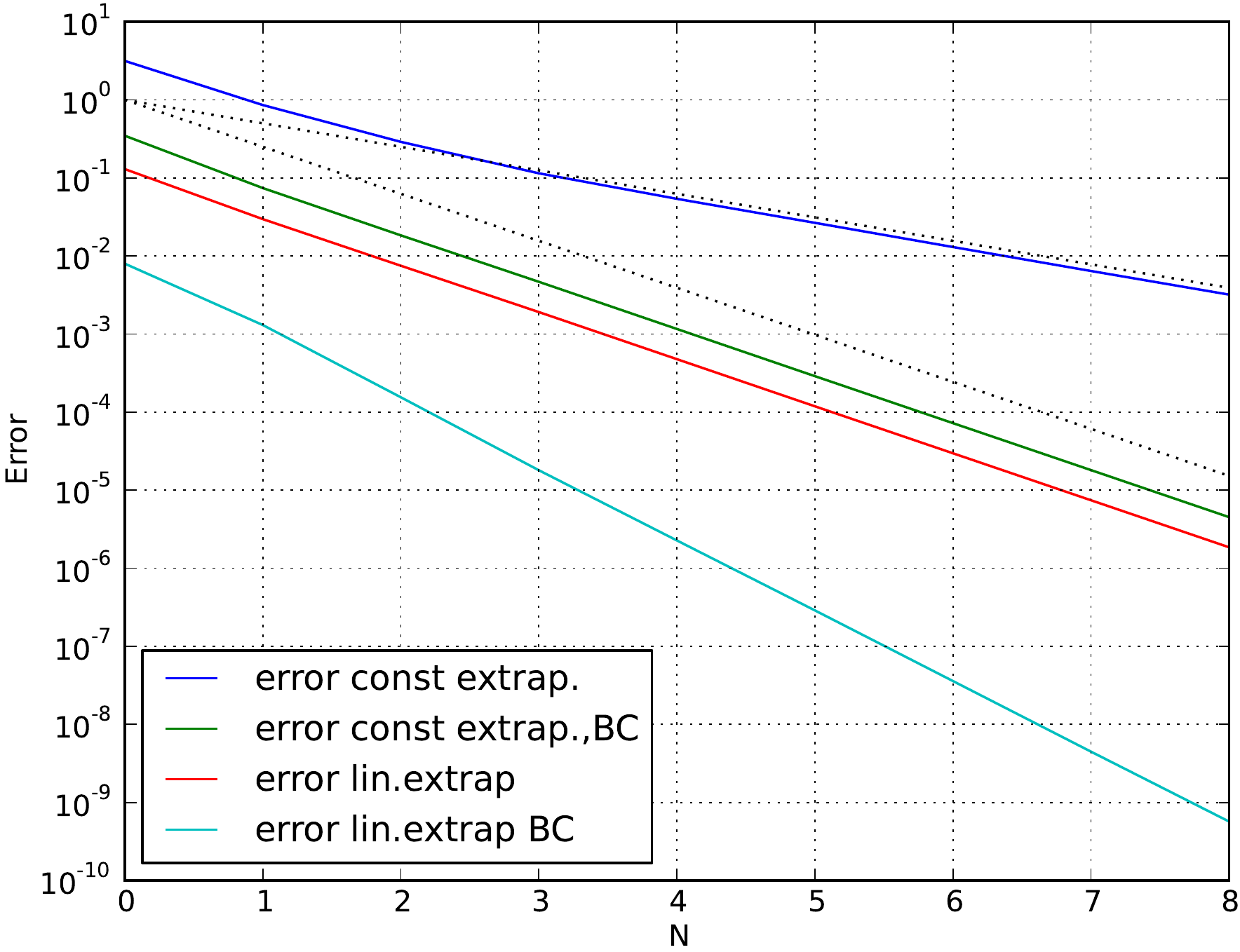}
}
\caption{\label{ConvergencePlot} Convergence of the system \eqref{springMass} in the cosimulation scheme, varying the exchange step size $H = 0.2\left(\frac{1}{2}\right)^N$.    Convergences as predicted are achieved - even more, balance corrected scheme for this problem converges of one order higher than proven.}
\end{figure}
The numerical convergence examination backs up the results from section \ref{convergence}, estimate \eqref{convErrBoundingFunction},  and equation \eqref{convergenceBC} -- even more, balance corrected scheme for this problem converges of one order higher than proven there. A sharper theoretical result should be achievable.

\begin{figure}%[h!tb]
\includegraphics[ width=0.5\textwidth]{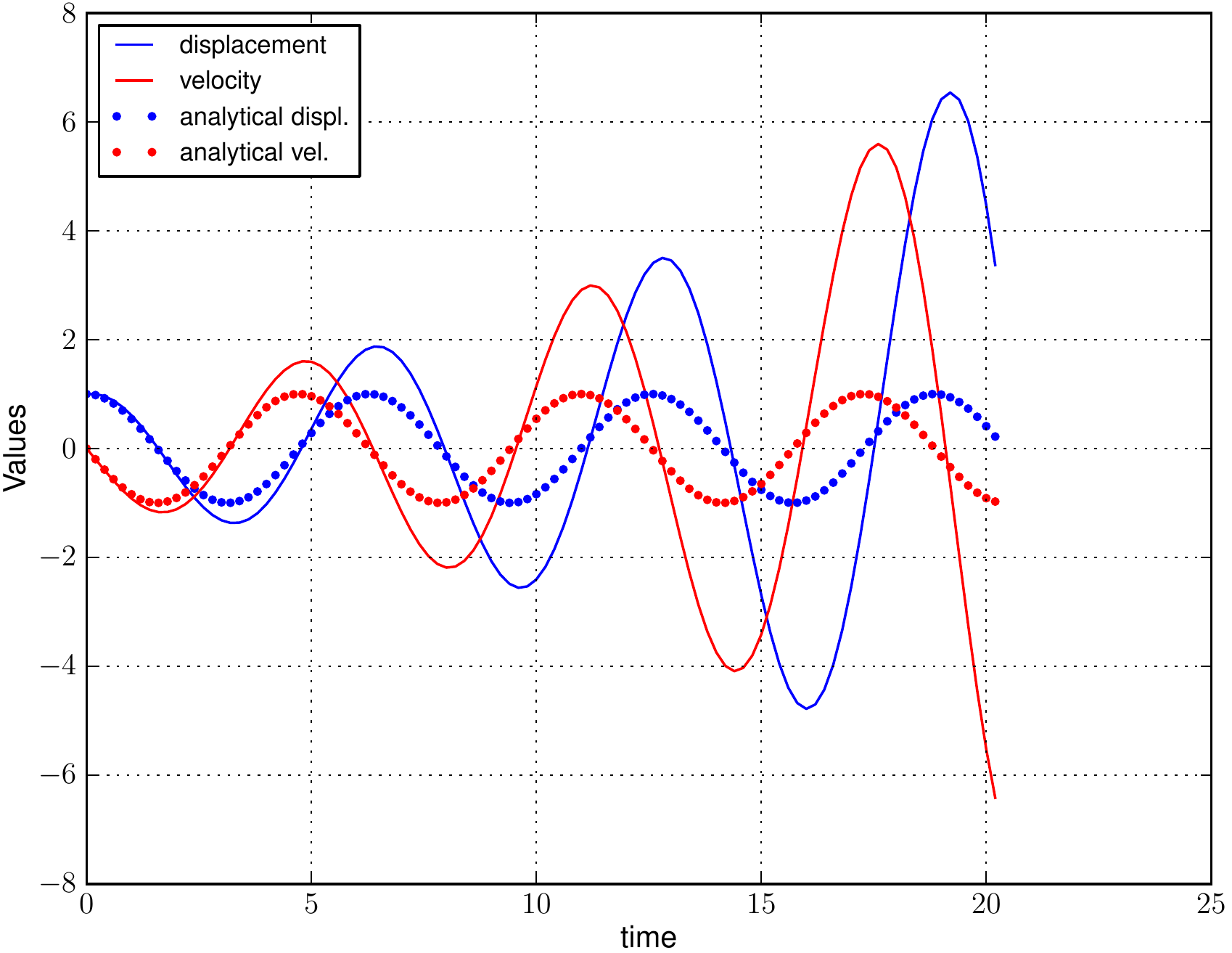}
\includegraphics[ width=0.5\textwidth]{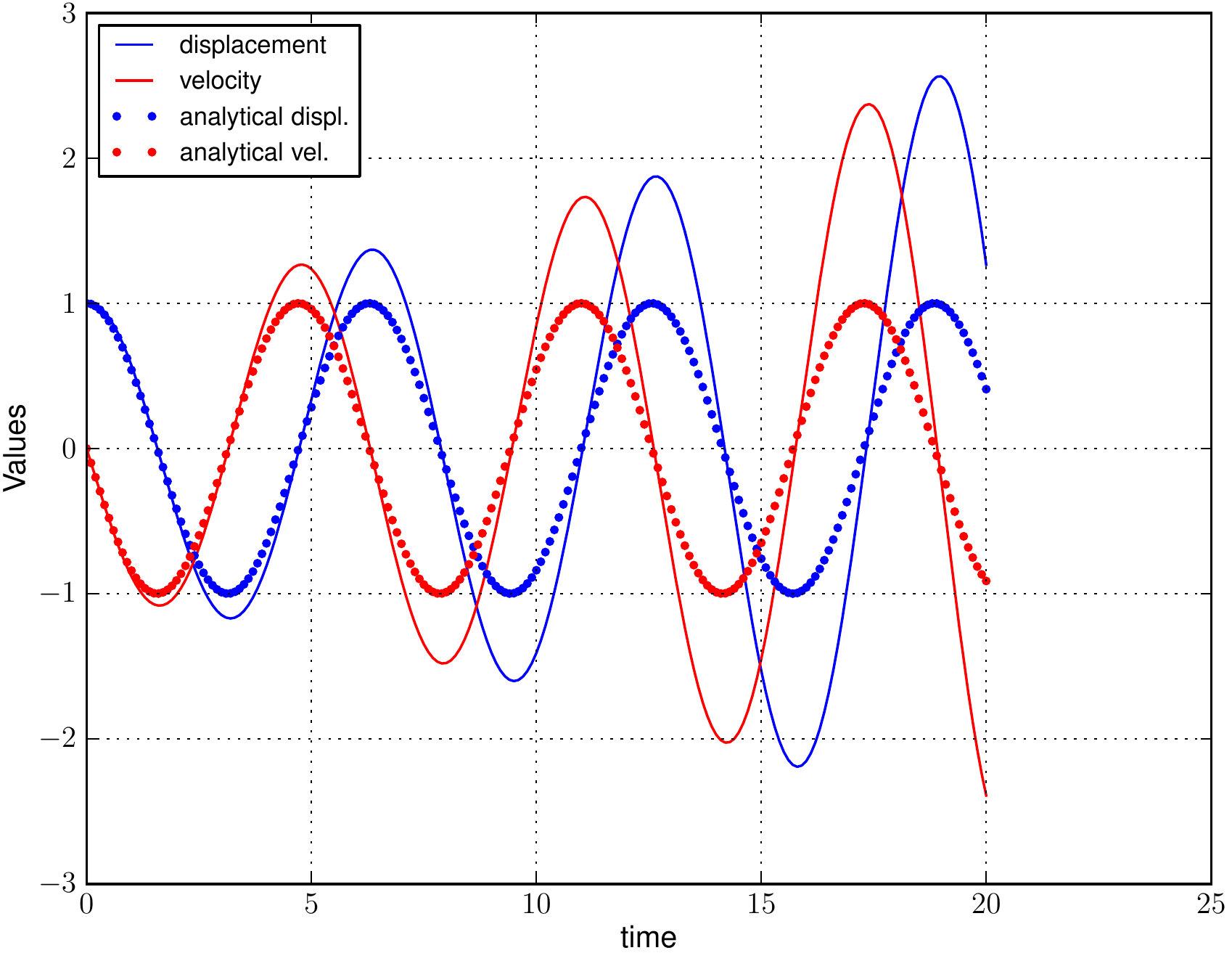}\\
\includegraphics[ width=0.5\textwidth]{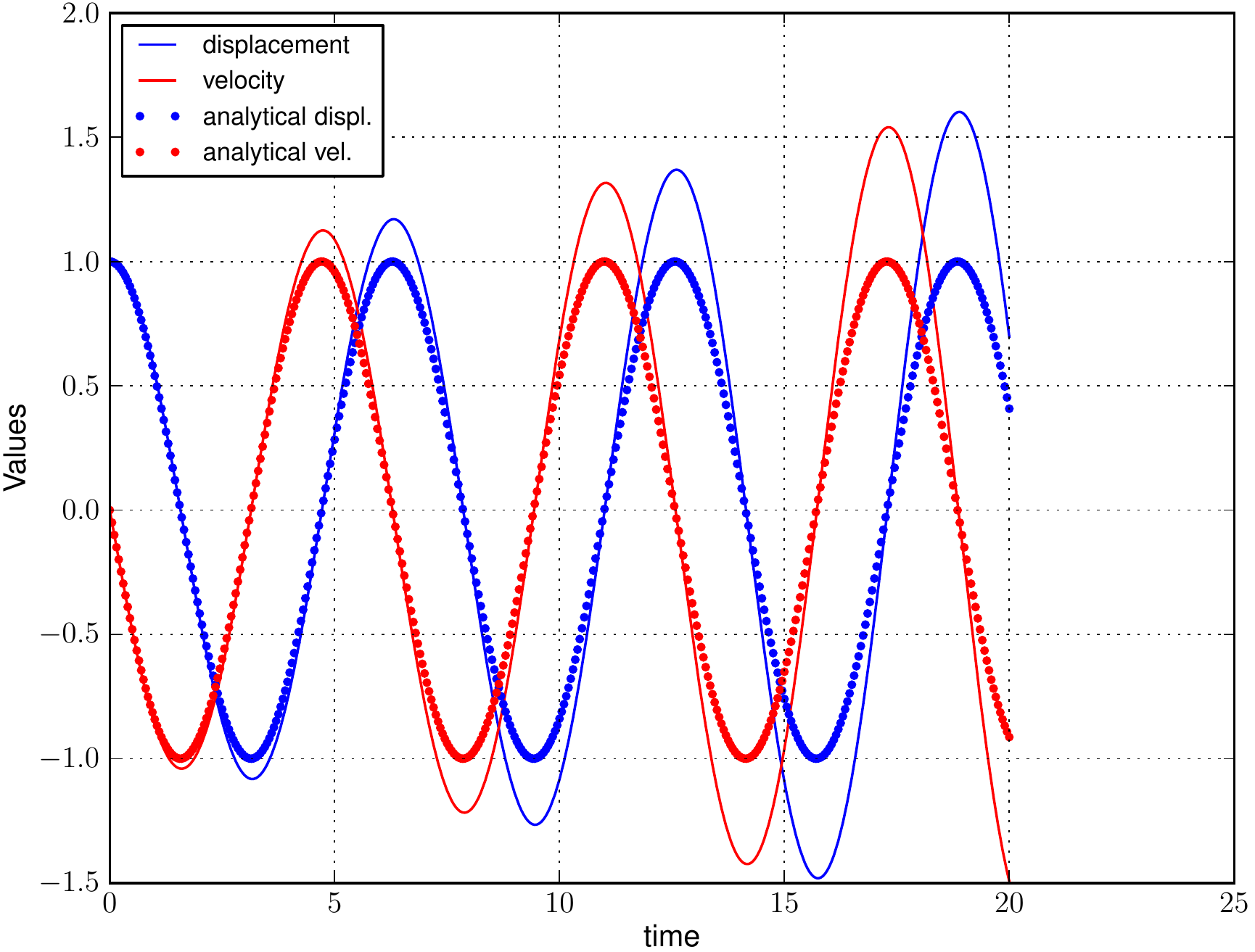}
\includegraphics[ width=0.5\textwidth]{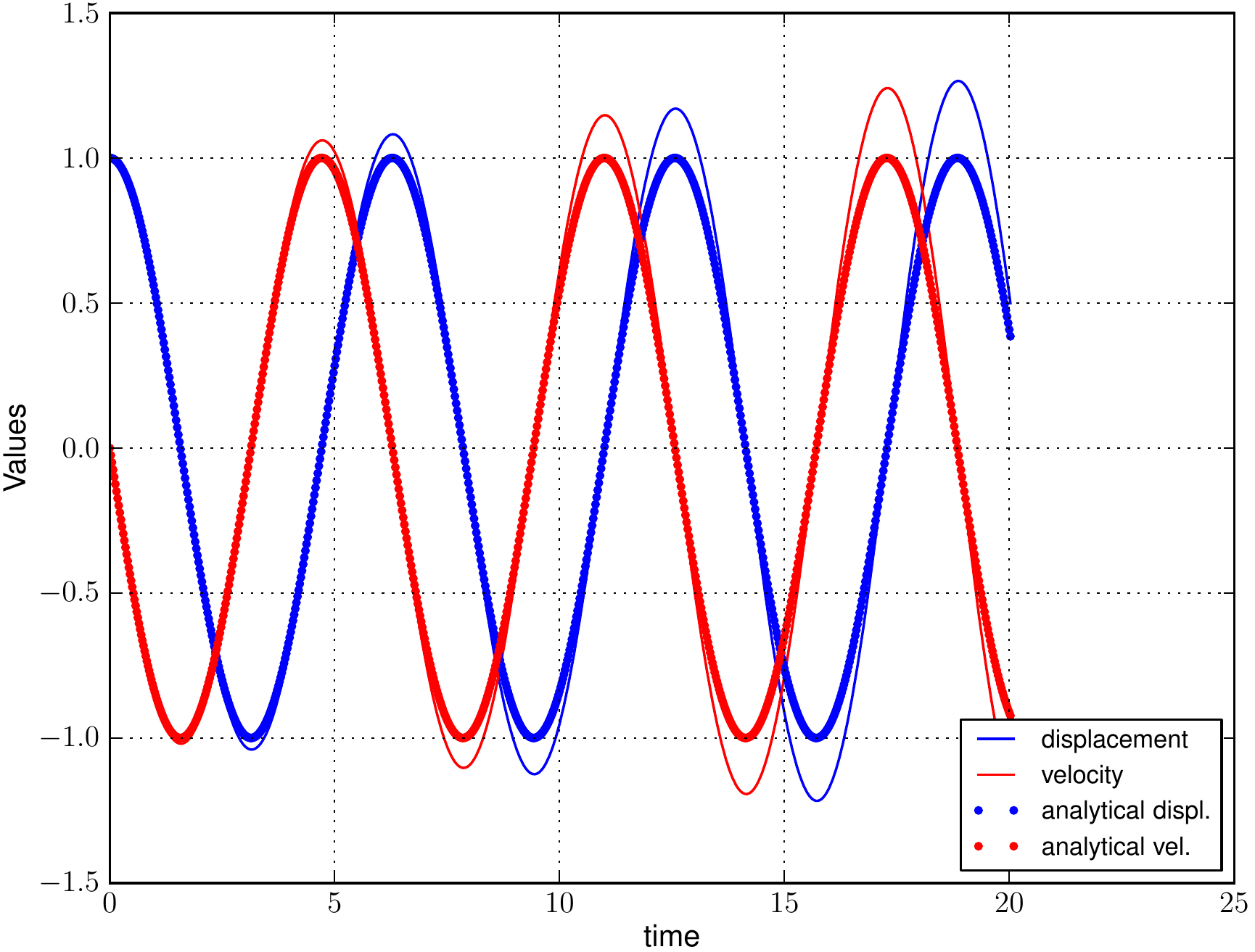}\\
\caption{\label{ConstCosimConvergencePlot} Simulation of the system \eqref{springMass} in the cosimulation scheme with constant extrapolation, varying the exchange step size $H$.  Upper row, left: $H = 0.2$, right:  $H = 0.1$, lower row: left: $H = 0.05$, right:  $H = 0.025$. Convergence of order $H$ is given, but there is no stability for any step size in sight. Figure previously published in \cite{ScharffMoshagen2017}.}
\end{figure}
\begin{figure}%[h!tb]
\includegraphics[ width=0.5\textwidth]{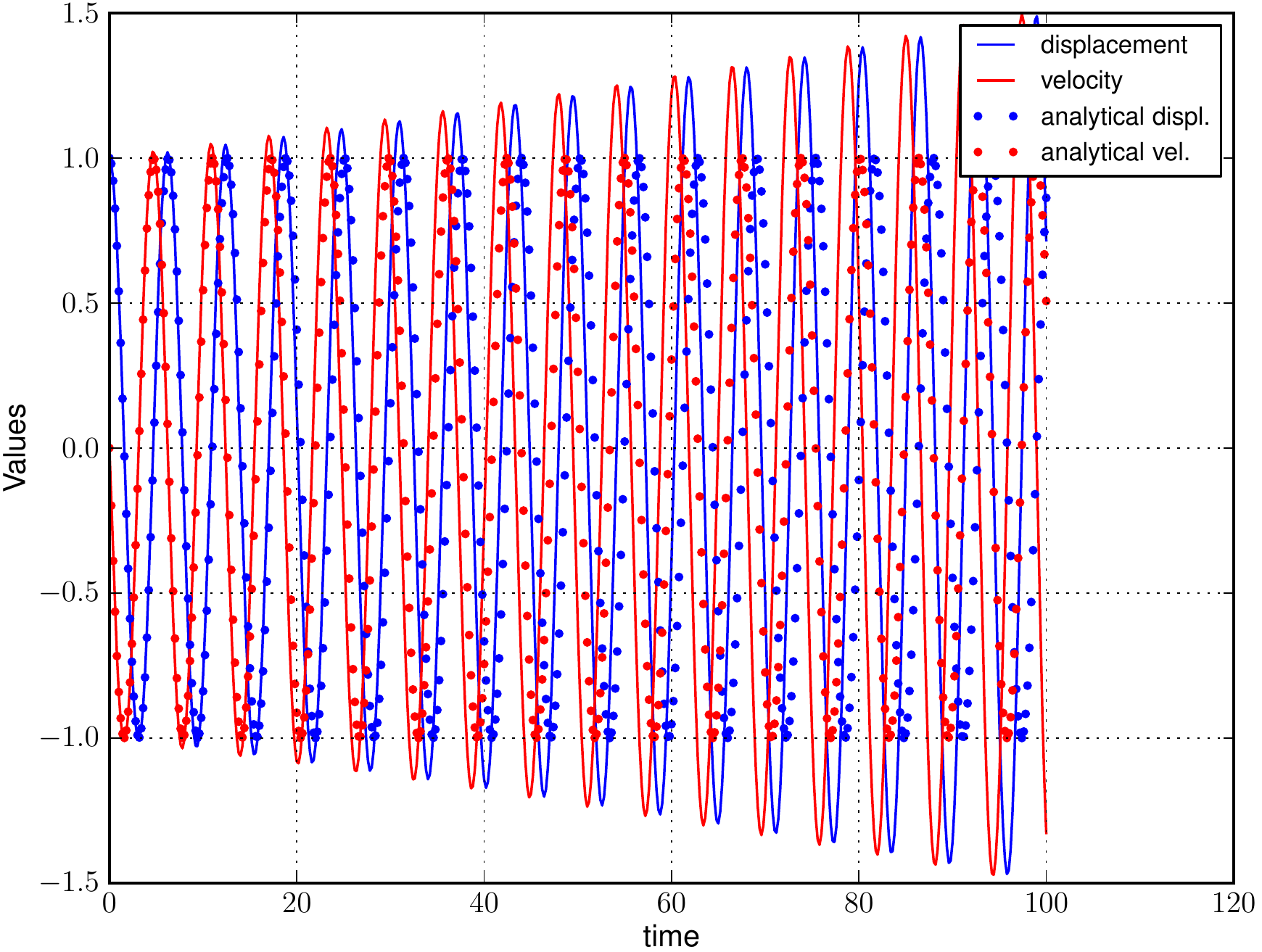}
\includegraphics[ width=0.5\textwidth]{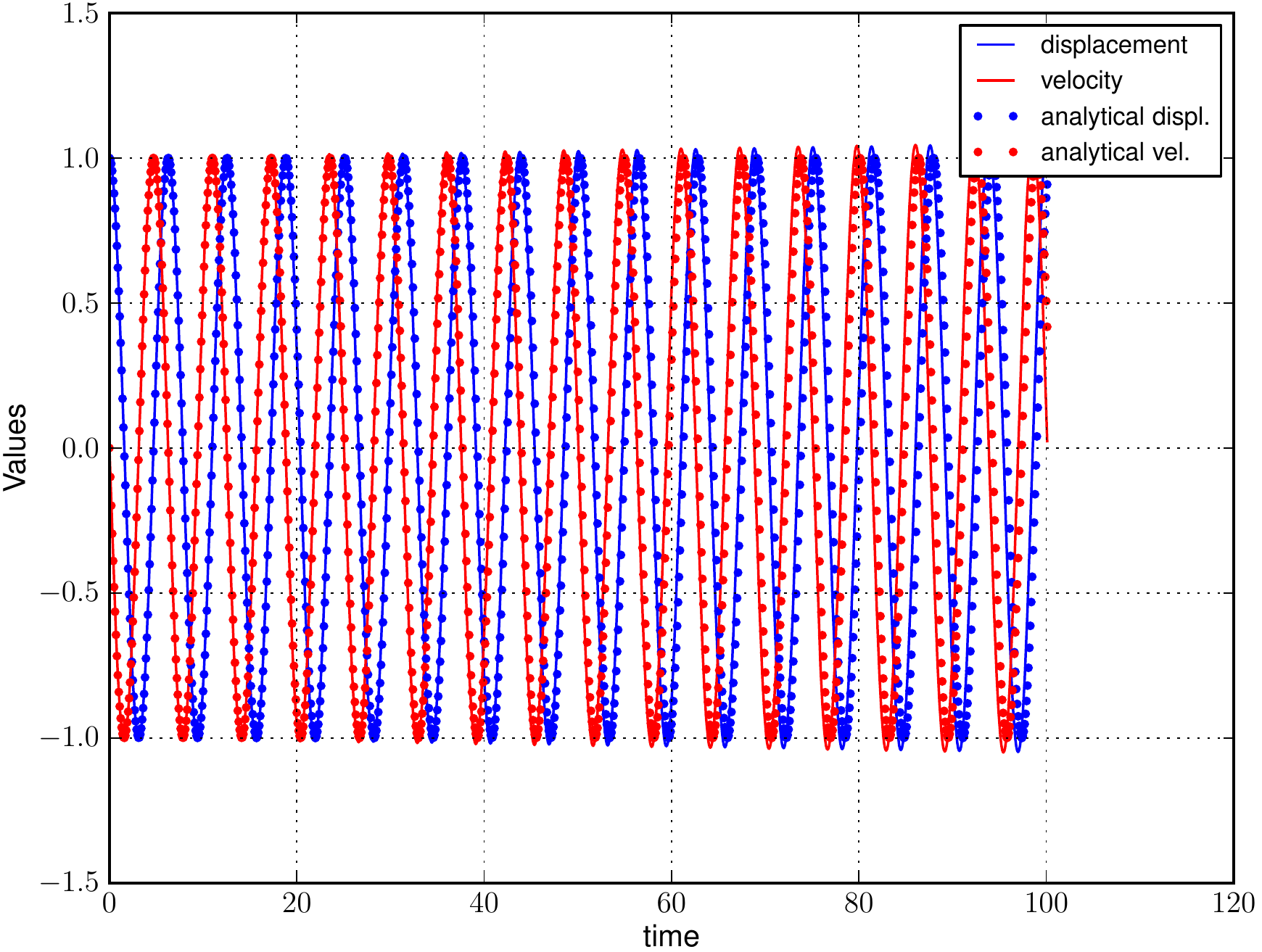}\\
\caption{\label{LinCosimConvergencePlot} Simulation of the system \eqref{springMass} in the cosimulation scheme with linear extrapolation, varying the exchange step size $H$.   Left: $H = 0.2$, right:  $H = 0.1$. $H = 0.05$   $H = 0.025$ are not shown because error is not visible in plot. Convergence of order $H^2$ is given, but as in the constant extrapolation case there is no stability. Left figure previously published in \cite{ScharffMoshagen2017}.}
\end{figure}

 But  the method is unstable for its explicite contributions, as proven in section \ref{Stability}. This means $\norm{\vec x}\longrightarrow \infty$ for $t\longrightarrow \infty$. The energy of our system is $E=\frac{1}{2}m v^2+ \frac{1}{2}cs^2=\scalar{\vec x}{\vec x}_{\frac{1}{2}\operatorname{diag}(m,c)}$, which is an equivalent norm, so lack of stability is equivalent to energy augmentation. \\  
This lack of stability can be interpreted in physics \todo{Really always this?} as a consequence of extrapolation errors in factors of power acting on subsystems boundaries:
 In \cite{ScharffMoshagen2017} the problem arose 
that errors in the force $y_1$ made during data exchange lead to errors in the power that acts on the mass. The system picks up energy and behaves unstable (See figure \ref{ConstCosimConvergencePlot}). This led to a reclassification to balance errors: Cumulated extrapolation errors in input variables that in fact are conserved quantities, and those in input variables that are factors of conserved quantities and such disturb a balance \cite[Section 3.2]{ScharffMoshagen2017}.

%in energyBalancing: In the following sections, it will be shown that balance correction techniques applied to the impulse as the integral of the force do not prevent this effect, as the a posteriori refeed of force then acts at another system state than it should, as the states have changed meanwhile - here the mass has changed its velocity.\\

\subsubsection{Pitfall %s}\subsubsection{
Subsystems methods}
\label{pitfallMultistep}
\begin{figure}%[h!tb]
\includegraphics[ width=0.8\textwidth ]%,trim={0 6cm  0 4cm},clip
{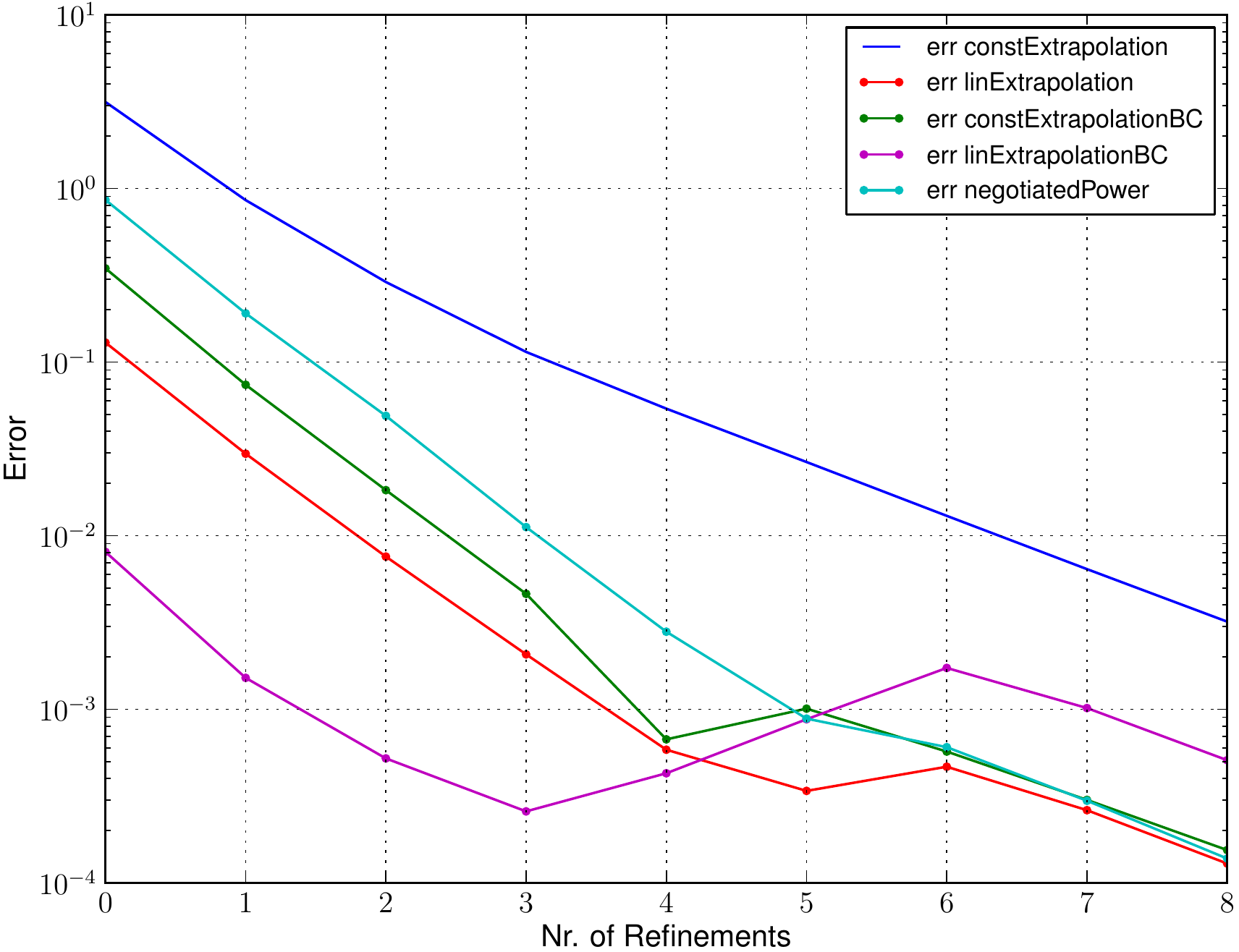}
\caption{\label{comparisonPlotVode} Convergence, $T_\text{end}=20$, subsystems refinement decisions left to subsystems solvers, \emph{vode} used on subsystems. In this setting, extrapolation error dominates at the beginning, then multistep methods initial value solvers convergence rate spoils convergence rate.   }
\end{figure}
Multistep methods need additional initial values: an $n$-step methods needs $n$ of them. Those initial values are usually calculated using one-step methods, possibly of lower order. Its error usually is not visible in the result, as it occurs only once in the calculation. But as subsystem solver restarts after data exchange, the error occurs $N$ times in cosimulation. Assume that the one-step method is of order $p'$. Then even if $H=ch$ holds, sum of its  error is $NCh^{p'}=NHCh^{p'-1}/c=(T_N-T_0)Ch^{p'-1}=O(h^{p'-1})$. In figure \ref{comparisonPlotVode}, this is $0(h)$, as visible for the last two refinements. The local minimum in error is due to different signs of the two contributions $O(h^{p'-1})$ and $O(H^{P+1})$.